\documentclass[11pt]{amsart}
\usepackage{cite}
\usepackage{wrapfig, lipsum,booktabs}
\usepackage{graphicx}
\usepackage{amssymb}
\usepackage{epstopdf}
\usepackage{verbatim}
\usepackage{bm}
\usepackage{multicol}
\usepackage{multirow}
\usepackage{subfigure}
\usepackage{fancyhdr} 
\usepackage{float}
\usepackage{stmaryrd}
\usepackage{color}
\usepackage{soul}
\usepackage{tikz}
\usepackage{todonotes}
\usetikzlibrary{patterns}
\usepackage{pgfplots}
\usepackage{cancel}
\newtheorem{theorem}{Theorem}[section]
\newtheorem{lemma}{Lemma}[section]
\newtheorem{remark}{Remark}[section]

\usepackage{multirow}
\usepackage{amsmath,amssymb,eucal}
\usepackage{graphicx,subfigure,epsfig}
\usepackage{psfrag}
\usepackage{url}
\usepackage[top=1in, bottom=1.in, left=1in, right=1in]{geometry}

\newcommand{\bld}[1]{\hbox{\boldmath$#1$}}    
\newcommand{\Th}{\mathcal{T}_h}
\newcommand{\pol}{\mathcal{P}}

\newcommand{\Eh}{\mathcal{E}_h}

\setulcolor{red}

\newcommand{\tng}{\mathsf{tng}}
\newcommand{\nrm}{\mathsf{nrm}}

\begin{document}
\title[HDG-FSI]{Arbitrary Lagrangian-Eulerian hybridizable 
discontinuous Galerkin methods for fluid-structure interaction}
\author{Guosheng Fu}
\address{Department of Applied and Computational Mathematics and 
Statistics, University of Notre Dame, USA.}
\email{gfu@nd.edu}
 \thanks{We acknowledge the partial support of this work
 from U.S. National Science Foundation through grant DMS-2012031.}

\begin{abstract}
  We present a novel (high-order) 
  hybridizable discontinuous Galerkin (HDG) scheme 
  for the fluid-structure interaction (FSI) problem.
The (moving domain) incompressible Navier-Stokes equations are discretized using
a divergence-free HDG scheme within the arbitrary Lagrangian-Euler (ALE)
framework. 
The nonlinear elasticity equations are discretized using a novel 
HDG scheme with an $H(\mathrm{curl})$-conforming velocity/displacement
approximation.
We further use a combination of the Nitsche's method (for the tangential component) 
and the  mortar method (for the normal component) to enforce  
the interface conditions on the fluid/structure interface. 
A second-order backward difference formula (BDF2) is use for the temporal discretization. 
Numerical results on the classical benchmark problem by Turek and Hron
\cite{Turek06} show a good performance of our proposed method.


\end{abstract}
\maketitle

\section{Introduction}
\label{sec:intro}
Fluid-structure interaction (FSI) describes a multi-physics phenomenon that
involves the highly non-linear coupling between a deformable or moving structure
and a surrounding or internal fluid. There has been intensive interest in
numerically solving FSI problems due to its wide applications in biomedical, engineering and
architecture fields.
\cite{bungartz2006fluid,chakrabarti2005numerical,dowell2001modeling,hou2012numerical,richter2017fluid}.

In this contribution, we present a novel hybridizable discontinuous Galerkin
(HDG) scheme to solve the 
nonlinear FSI problem modeled by the
incompressible Navier-Stokes equations in the fluid domain and the equations for
nonlinear hyperelasticity in the structure domain.
The fluid problem is discretized on the moving domain with an arbitrary
Lagrangian-Eulerian divergence-free HDG scheme, which is similar to the 
ALE $H(\mathrm{div})$-conforming HDG method
of Neunteufel and Sch\"oberl  \cite{Neunteufel21}, and the structure problem is discretized using
a novel $H(\mathrm{curl})$-conforming HDG scheme, which is closely related to
the TDNNS method of Pechstein and Sch\"oberl \cite{PechsteinSchoberl11}. 
The interface coupling terms are treated naturally via a combination of the
Nitsche's technique in the tangential component and the mortar method (via
a Lagrange multiplier) in the normal component.
We then apply a second order backward difference formula  (BDF2) for the temporal
discretization.
One salient feature of our method is that the fluid velocity approximation
maintains to be exactly 
divergence-free throughout.

Let us make two comments on our choice of the HDG scheme for on the 
first order system formulation, denoted as {\sf L-HDG} since is it based on a
local discontinuous Galerkin formulation, 
over another popular formulation for on the original second order equations, 
denoted as {\it IP-HDG} since is is based on an interior penalty DG formulation.
\begin{itemize}
  \item [(1)]
    {\it  {\sf L-HDG} has an advantage over {\sf IP-HDG}  in terms of the
    choice of the  stabilization parameter.}
In the proposed first order system formulation,
the semidiscrete stability is ensured as long as the HDG stabilization parameters are
positive. We take the stabilization parameter $\alpha$ to be of order {\it one}, in
particular, we use $\alpha=2\mu^f$ in the fluid solver and the Nitsche coupling
term, and 
$\alpha=2\mu^s$ in the structure solver, where $\mu^f$ is the dynamic viscosity
and $\mu^s$ is the Lam\'e's second parameter. 
Note that the original $H(\mathrm{div})$-conforming HDG method \cite{LehrenfeldSchoberl16}
and the Nitsche's technique is based on a second order PDE and interior penalty, where one has to use
a {\it sufficiently large} stabilization/penalty parameter that depends on both the mesh size and the polynomial
degree. While it is easy to estimate the minimal stabilization parameter that
guarantees the convergence of the HDG scheme \cite{LehrenfeldSchoberl16} for the
{\it linear} Stokes operator, such an estimation is expected to be very hard for
the nonlinear elasticity operator, which depends on the maximum
eigenvalue of a nonlinear fourth order elasticity tensor.
\item [(2)]
    {\it For incompressible Navier-Stokes, 
      {\sf L-HDG} usually poses a superconvergence property (via postprocessing)
      in the diffusion dominated regime, see, e.g., \cite{CFQ18,FJQ19,CF18, TNBP19},  and provide optimal convergence in the convection
    dominated regime.} However, this superconvergence property is lost for the
    original {\sf IP-HDG} scheme.
Although the concept of {\it projected jumps}
\cite{Lehrenfeld10,LehrenfeldSchoberl16} can be used for 
{\sf IP-HDG} to restore
superconvergence for the Stokes problems, where polynomials of a lower degree is
used for the {\it hybrid} facet unknowns.
This technique might lead to accuracy loss for an (implicit) HDG discretization of the
Navier-Stokes problem in the convection-dominated regime \cite{Lehrenfeld10}.
\end{itemize}
We note that superconvergent postprocessing is not exploited in this work.

The rest of the paper is organized as follows.
In Section \ref{sec:fluid}, we introduce the ALE-divergence-free HDG scheme for 
the moving domain incompressible Navier-Stokes equations.
In Section \ref{sec:elas}, we introduce the $H(\mathrm{curl})$-conforming HDG scheme for 
the equations for nonlinear elasticity.
We then introduce an HDG method to solve the FSI problem by combing the above
two HDG schemes in Section \ref{sec:ale}.
Numerical results on a classical benchmark problem are presented in Section \ref{sec:num}.
We conclude in Section \ref{sec:conclude}.

\section{The ALE-divergence-free HDG scheme for 
incompressible Navier-Stokes}
\label{sec:fluid}
In this section, we introduce the ALE-divergence-free HDG scheme 
for the incompressible Navier-Stokes equations.
We largely follow the work \cite{Fu20,Neunteufel21} to derive the method.

\subsection{The ALE-Navier-Stokes equations}
Consider the Navier-Stokes {equations}
on a moving domain $\Omega^f_t\subset\mathbb{R}^d$, 
$d\in\{2,3\}$, for $t\in [0, T]$, 
given by a continuous ALE map \cite{Nobile01,Quaini08}: 
\begin{align}
  \label{ale}
  \bld \phi_t:\Omega_0^f\subset \mathbb{R}^d\longrightarrow \Omega^f_t, 
  \quad\quad \mathbf x(\mathbf x_0, t) = \bld \phi_t(\mathbf x_0),\quad  
  \forall t\in [0,T],
\end{align}
where $\Omega^f_0$ is the initial (bounded) fluid domain at time $t=0$ with possibly
curved boundaries.
In this section we assume that the ALE map $\bld \phi_t$ in \eqref{ale}
is {\it apriori} given to simplify the presentation.

The Navier-Stokes equations in ALE non-conservative form 
\cite{Nobile01,Quaini08} is given as follows: 
\begin{subequations}
\label{ns-eq-ale}
  \begin{alignat}{2}
\label{ns-eq-ale1}
  \rho^f\left.\frac{\partial \bld u^f}{\partial t}\right|_{\mathbf x_0} 
  +\rho^f(\bld u^f-\bld \omega^f)\cdot\nabla_{\mathbf x}\bld u^f
  -\mathrm{div}_{\mathbf x}(2\mu\mathbf{D_x}(\bld u^f)  -p^f\bld {I})
    =&\; \rho^f\bld f^f, \quad&& \text{in} \;\Omega_t^f\times [0,T]\\
\label{ns-eq-ale2}
    \mathrm{div}_{\mathbf x}\bld u^f = &\;0,\quad     
                                       && \text{in} \;\Omega_t^f\times [0,T]
\end{alignat}
\end{subequations}
where $\mathbf D_{\mathbf x}$ is the symmetric strain rate tensor
\begin{align*}
  \mathbf {D_x}(\bld u) = \frac12(\nabla_{\mathbf x} \bld u + 
  (\nabla_{\mathbf x} \bld u)^T),
\end{align*}
$\bld I$ is the identity tensor, 
$\bld u^f(\mathbf x,t)$ is the fluid velocity field, $p^f(\mathbf x, t)$ is 
the pressure, 
$\rho^f$ is the (constant) fluid density, $\mu^f$ is the (constant) coefficient of 
dynamic viscosity, $\bld f^f$ is the body forces, and
\begin{align}
  \label{domain-vel}
  \bld \omega^f(\mathbf x, t) =
  \left.\frac{\partial \mathbf x}{\partial
  t}\right|_{\mathbf x_0} = 
  \frac{\partial \bld \phi_t}{\partial
  t}\circ\bld \phi_t^{-1}(\mathbf x) 
\end{align}
is the domain velocity.
We assume the Navier-Stokes equations \eqref{ns-eq-ale} is further equipped with
the following homogeneous Dirichlet boundary condition:
\begin{align}
  \label{dirichlet}
  \bld u^f(\mathbf x,t) = 0, \quad  \forall \mathbf x\in \partial\Omega^t, \quad t\in[0,T].
\end{align}

\subsection{Mesh and finite element spaces}
\subsubsection{Mesh and mappings}
Let $\Th^{f,0}:=\{K^0\}$ be a conforming simplicial
triangulation\footnote{
  The restriction to simplicial meshes is not essential, which
  is only for the purpose of simple presentation of the
  finite element spaces. 
  One can easily work with conforming hybrid meshes that consist of triangles/quadrilaterals in two dimensions, or 
tetrahedra/prisms/pyramids/hexahedra in three dimensions.
See, e.g., \cite{Zaglmayr06,FuentesKeithDemkowiczNagaraj15} for various high-order finite element spaces 
on hybrid meshes.
} of the
initial
fluid domain $\Omega^f_0\subset \mathbb{R}^d$, 
where the element $K^0=\Phi_K^0(\widehat K)$ is a mapped
simplex from the reference simplex element 
\[\widehat K:=\{\widehat{\mathbf x}=(\widehat x_1, \cdots, \widehat x_d):
  \quad  \widehat x_i \ge0,\;\forall 1\le i\le d,\quad
    \sum_{i=1}^d\widehat x_i \le 1.
\}
\]
We assume the mapping 
$\Phi^0_K\in [\pol^m(\widehat K)]^d$, 
where $\pol^m(\widehat K)$ is the space of polynomials of degree at most $m\ge 1$
on the reference element $\widehat K$.
Moreover, let $\Th^{f,t}:=\{K^t=\bld \phi_t(K^0)\}$ be the mapped triangulation of the
deformed domain $\Omega_t^f$ at time $t$, where $\bld \phi_t$ is the ALE map
given in \eqref{ale}. 
Denoting the composite mapping $\Phi_K^t=\bld\phi_t\circ\Phi_K^0: 
\widehat{K}
\rightarrow K^t$, we have $K^t=\Phi^t_K(\widehat K)$. 
We assume the mesh $\Th^{f,t}$
is regular in the sense that no elements with a degenerated or negative Jacobian determinant
exist. Without loss of generality, we further assume the ALE map \eqref{ale}
is  a continuous piecewise polynomial of degree $m$: 
\begin{align}
  \label{ale-pm}
  \bld \phi_t\in \bld S_h^m:=\{\bld v \in [H^1(\Omega^{f}_0)]^d:\quad 
    \bld v\circ \Phi_K^0(\widehat{\mathbf x})\in 
    [\pol^m(\widehat K)]^d, \quad \forall K^0\in \Th^{f,0}
  \}.
\end{align}


For the reference $d$ dimensional simplex element $\widehat K$, we denote
$\partial \widehat K$ as its boundary, which consists of $(d+1)$ 
boundary facet $\{\widehat E_{l}:=\widehat \Psi_l(\widehat E)
\}_{l=1}^{d+1}$, where $\widehat \Psi_l$ is the affine mapping
from the reference $(d-1)$ dimensional simplex element $\widehat E$
to the boundary facet $\widehat E_{l}$, where
\[
  \widehat E:=\{\widehat{\mathbf z}=(\widehat z_1,\cdots, \widehat z_{d-1}):\;\;
  \quad  \widehat z_i \ge0,\;\forall 1\le i\le d-1,\quad
  \sum_{i=1}^{d-1}\widehat z_i \le 1.
  \}
\] 

We need the following 
Jacobian matrices and their determinants:
\begin{subequations}
  \label{jac}
  \begin{align}
  F_K^t :=& \nabla_{\widehat {\mathbf x}}\Phi_K^t\in\mathbb{R}^{d\times d}, \quad 
  J_K^t := \mathrm{det}(F_K^t),\\
  {F^t_0} :=& \nabla_{\mathbf x_0}\bld \phi_t\in\mathbb{R}^{d\times d}, \quad 
  {J^t_0} := \mathrm{det}({F}^t_0),\\
    \widehat F_l:=&\nabla_{\widehat{\mathbf z}}\widehat\Psi_l
    \in\mathbb{R}^{d\times(d-1)},\quad 
  \widehat{J}_l := \sqrt{\mathrm{det}\left((\widehat F_l)^T\widehat
  F_l\right)}.
\end{align}
\end{subequations}
A simple application of the chain rule implies that 
\[
F_K^t = F_0^t\, F_K^0, \quad J_K^t = J_0^tJ_K^0.
\]

We denote $\partial \Th^{f,t}:=\{\partial K^t\}$ as the collection of element
boundaries of the mesh $\Th^{f,t}$, where 
$\partial K^t=\{E_{K,l}^t\}_{l=1}^{d+1}$ is the boundary of element 
$K^t$, with $E_{K,l}^t:=\Phi_{K}^t(\widehat E_l)$ being the mapped facet.
Let $\widehat {\bld n}_l$ be the normal direction of the 
reference boundary facet $\widehat E_l$, and $\bld n_{K,l}^t$ be the normal 
direction of the physical boundary facet $E_{K,l}^t$. There holds 
\[
  \bld n_{K,l}^t
  \circ(\Phi_K^t)^{-1}
  = \frac{(F_K^t)^{-T}\,\widehat{\bld n}_l}{\|(F_K^t)^{-T}\,\widehat{\bld n}_l\|}.
\] 

We denote 
$\Eh^{f,t}:=\{E^t\}$ as the mesh skeleton of $\Th^{f,t}$,
which consists of all the facets.
Here $E^t:=\Psi_E^t(\widehat E)$ with
the mapping $\Phi_E^t$  constructed by composition:
for a facet $E^t=E_{T,l}^t$ that is the $l$-th boundary facet of an element 
$K^t$, we denote $\Psi_E^t:\widehat E\rightarrow E^t$ as
\[
\Psi_E^t:=\Phi_K^t\circ\widehat \Psi_l.
\]
We note that if the facet $E^t$ happens to be an interior facet which is also
the $m$-th boundary facet of another element $\widetilde K$, then there holds
\[
  \Psi_E^t=\Phi_K^t\circ\widehat \Psi_l=\Phi_{\widetilde K}^t\circ\widehat
  \Psi_m.
\] 
Hence, the facet mapping $\Psi_E^t$ is uniquely determined as
it does not depend on which associated volume element of $E^t$ is used.
We now denote the following surface 
Jacobian matrix, its Moore-Penrose pseudo inverse and determinant for the
mapping $\Psi_E^t$:
\[
  F_E^t:=\nabla_{\widehat{\mathbf z}}\Psi_E^t
  \in\mathbb{R}^{d\times
  (d-1)},\quad 
  (F_E^t)^{-1}:=
  \left((F_E^t)^TF_E^t\right)^{-1}F_E^t\in\mathbb{R}^{(d-1)\times
  d},
  \quad
  {J}_E^t := \sqrt{\mathrm{det}\left((\widehat F_E^t)^TF_E^t\right)}. 
\] 

\newcommand{\Whf}[1]{W_h^{#1, f}}
\newcommand{\Vhf}[1]{\bld V_h^{#1, f}}
\newcommand{\Shf}[1]{\bld \Sigma_h^{#1, f}}
\newcommand{\tWhf}[1]{\widetilde W_h^{#1, f}}
\newcommand{\tVhf}[1]{\widetilde{\bld V}_h^{#1, f}}

\newcommand{\Vhs}[1]{\bld V_h^{#1, s}}
\newcommand{\Shs}[1]{\bld \Sigma_h^{#1, s}}
\newcommand{\tVhs}[1]{\widetilde{\bld V}_h^{#1, s}}

\subsubsection{The finite element spaces}
We first introduce the following (discontinuous) 
finite element spaces on the mesh $\Th^{f,t}$: 
\begin{subequations}
  \label{spaces}
\begin{align}
  W_h^{k, f}:=
  &\;\{w\in L^2(\Omega_t^f): \;\;
    w|_{K^t} = \widehat{w}\circ (\Phi_K^t)^{-1},\;\; 
    \widehat{w}\in \pol^k(\widehat K), 
\;\forall K^t\in \Th^{f,t}\}, \\ 
    \bld V_h^{k,f}:=&\;\left\{\bld v\in [L^2(\Omega_t^f)]^d: \;\;
      \bld v|_{K^t} = \frac{1}{J_K^t}F_K^t \left(\widehat {\bld
      v}\circ(\Phi_K^t)^{-1}\right), 
      \;\; \widehat{\bld v}\in [\mathcal{P}^k(\widehat K)]^d, \;\forall
    K^t\in \Th^{f,t}\right\}, \\
  \bld\Sigma_h^{k, f}:=
                    &\;\{\bld\sigma\in [W_h^{k,f}]^{d\times d}: \;\;
                   \bld\sigma   \text{ is symmetric}
                    \}.
    \end{align}

Note that the standard pull-back mapping is used to define the
scalar finite element space $W_h^{k,f}$ and the symmetric tensor finite element space 
$\bld \Sigma_h^{k,f}$, 
which will be used to approximate 
the pressure field $p^f$, and the strain rate tensor
$\mathbf D_{\mathbf
x}(\bld u^f)$, respectively.
While the Piola mapping is used to define the vector
finite element space $\bld V_h^{f,k}$, which will be used 
to approximate the fluid velocity $\bld u^f$. 
The use of Piola mapping in $\bld V_h^{k,f}$ ensures a 
strong mass conservation 
for the HDG scheme \eqref{hdg-fluid}
on curved meshes, see Lemma \ref{lma:div-free} below.

We also need the following (hybrid) finite element spaces on the mesh skeleton
$\Eh^{f,t}$:
    \begin{align}
  \label{vhath0}
          \widetilde{W}_{h}^{k,f}:=&\;\{\widetilde w\in 
            L^2(\Eh^{f,t}): \;\;
            \widetilde w|_{E^t} = \widehat {
      w}\circ(\Psi_E^t)^{-1},
      \;\;\widehat {w}\in
    \pol^k(\widehat E), \quad\forall E^t\in\Eh^{f,t}\}, \\ 
        \widetilde{\bld V}_h^{k,f}:=&\;\{\widetilde{\bld v}
          \in [L^2(\Eh^{f,t})]^d: \;\;\widetilde{\bld v}|_{E^t}
          = \tng\left((F_E^t)^{-T} \,\widehat {\bld
          v}\circ(\Psi_E^t)^{-1}\right),
      \;\;\widehat {\bld v}\in
        [\pol^k(\widehat E)]^d,\quad \forall E^t\in\Eh^{f,t}\},
\end{align}
where $\tng(\bld v)|_{E} := \bld v - (\bld v\cdot \bld n_E)\bld n_E$
denotes the tangential component of the vector $\bld v$ on the facet $E$, whose
normal direction is $\bld n_E$.
Note that the normal component of functions in 
$\tVhf{k}$ vanishes on the whole mesh skeleton.
Here the standard pull-back mapping is used for 
the scalar skeleton space $\tWhf{k}$, which  will be used to approximate
 the normal-normal component of the stress,  
 $\bld n \cdot (2\mu_f\mathbf{D}_{\mathbf x}(\bld u^f)-p^f\mathbf I)\bld n$,  on
 the mesh skeleton,  and the covariant mapping is used for the 
 vector skeleton space $\tVhf{k}$, 
which preserves tangential continuity and 
 will be used to approximate the tangential component of
 fluid velocity, $\tng(\bld u^f)$,  on the mesh skeleton.
\end{subequations}

\subsection{The divergence-free HDG scheme: spatial discretization}
In this subsection, we focus on the spatial discretization of the 
equations \eqref{ns-eq-ale}. 
We work on the physical deformed domain $\Omega^f_t$ at a fixed time $t\in [0,
T]$. To introduce the scheme,  we first reformulate the equations \eqref{ns-eq-ale} to the following
first-order system:
\begin{subequations}
\label{ns-eq-reform}
  \begin{alignat}{2}
\label{ns-eq-reform1}
  \rho^f\left.\frac{\partial \bld u^f}{\partial t}\right|_{\mathbf x_0} 
    +\rho^f(\mathrm{div}_{\mathbf x}\bld \omega^f)\bld u^f
    +\mathrm{div}_{\mathbf x}(\rho^f(\bld u^f-\bld \omega^f)\otimes
    \bld u^f-\bld \sigma^f)
    =&\; \rho^f\bld f^f, \\
\label{ns-eq-reform2}
\bld \sigma^f-(2\mu^f\bld \epsilon^f-p^f\mathbf I)
    =&\; 0, \\
\label{ns-eq-reform3}
\bld \epsilon^f- \mathbf{D}_{\mathbf x}(\bld u^f)=&\; 0, \\
\label{ns-eq-reform4}
    \mathrm{div}_{\mathbf x}\bld u^f = &\;0.
\end{alignat}
\end{subequations}

Three local variables (defined on the mesh $\Th^{f,t}$), namely, the pressure $p_h^f$, velocity $\bld u_h^f$,
and strain rate tensor $\bld \epsilon_h^f$, 
and two global variables (defined on the mesh skeleton $\Eh^{f,t}$), 
namely the normal-normal component of the stress $\widetilde{\sigma}_h^f$, and the
tangential component of the velocity $\widetilde{\bld u}_h^f$ will be used in
our scheme.
We use polynomials of degree $k-1$ for the pressure 
approximation, and polynomials of degree $k\ge 1$ for the other variables, i.e., 
\[
  p_h^f\in \Whf{k-1}, \;\;
 \bld \epsilon_h^f\in\Shf{k},\;\;
 \bld u_h^f\in\Vhf{k},\;\;
 \widetilde{\sigma}_h^f\in\tWhf{k},\;\;
 \widetilde{\bld u}_h^f\in\tVhf{k},
\]

\newcommand{\intV}[3]{\left(#1, #2\right)_{\Th^{#3, t}}}
\newcommand{\intVX}[2]{\left(#1, #2\right)_{\Th^{f,0}}}
\newcommand{\intVo}[3]{\left(J_0^t#1, #2\right)_{\Th^{#3, 0}}}
\newcommand{\intF}[3]{\left\langle#1, #2\right\rangle_{\partial\Th^{#3, t}}}
\newcommand{\intVS}[2]{\left(#1, #2\right)_{\Th^{s}}}
\newcommand{\intFS}[2]{\left\langle#1, #2\right\rangle_{\partial\Th^{s}}}
\newcommand{\intG}[2]{\left\langle#1, #2\right\rangle_{\Gamma_t}}

The spatial discretization of our HDG scheme 
for the equations \eqref{ns-eq-reform} with homogeneous Dirichlet 
boundary conditions \eqref{dirichlet}
reads as follows:
Find 
$
(p_h^f,
 \bld \epsilon_h^f,
 \bld u_h^f,
 \widetilde{\sigma}_h^f,
 \widetilde{\bld u}_h^f)
 \in \Whf{k-1}\times\Shf{k}\times\Vhf{k}\times\tWhf{k}
 \times \tVhf{k}
 $ with $\widetilde{\bld u}_h^f|_{\Gamma_h^f}=0$
such that 
\begin{subequations}
  \label{hdg-fluid}
  \begin{align}
  \label{hdg-f1}
  \intV{
  \rho^f\left.\frac{\partial \bld u_h^f}{\partial t}\right|_{\mathbf x_0} 
  }{\bld v_h^f}{f}
  +\intV{
   \rho^f(\mathrm{div}_{\mathbf x}\bld \omega^f)\bld u_h^f
  }{\bld v_h^f}{f}
  -\intV{
    \rho^f(\bld u_h^f-\bld \omega^f)\otimes
    \bld u_h^f
  }{\nabla_{\mathbf x}\bld v_h^f}{f}
  \hspace{-0.1\textwidth}&\\
  +2\mu^f\intV{\bld \epsilon_h^f
  }{\nabla_{\mathbf x}\bld v_h^f}{f}
  -\intV{p_h^f
  }{\mathrm{div}_{\mathbf x}\bld v_h^f}{f}
  -\intF{\widetilde{\bld {Flux}}_{\mathrm v}-\widetilde{\bld {Flux}}_{\mathrm c}
}{\;\bld v_h^f}{f}&\;=\;
  \intV{\rho^f\bld f^f}{\bld v_h^f}{f}, \nonumber\\
  \label{hdg-f2}
 2\mu^f \intV{\bld \epsilon_h^f-\mathbf D_{\mathbf x}(\bld u_h^f)}{\bld G_h^f}{f}
  +2\mu^f\intF{\tng(\bld u_h^f-\widetilde{\bld u}_h^f)}{\bld G_h^f\bld n}{f}
                                                  &\;=\;0,\\
  \label{hdg-f3}
 \intV{\mathrm{div}_{\mathbf x}\bld u_h^f}{q_h^f}{f} &\;=\;0,\\
  \label{hdg-f4}
  \intF{\bld u_h^f\cdot\bld n}{\widetilde\tau_h^f}{f} &\;=\;0,\\
  \label{hdg-f5}
  \intF{\widetilde{\bld {Flux}}_{\mathrm v}-\widetilde{\bld {Flux}}_{\mathrm c}
  }{\;\tng(\widetilde{\bld v}_h^f)}{f} &\;=\;0,
  \end{align}
  for all 
$
(q_h^f,
 \bld G_h^f,
 \bld v_h^f,
 \widetilde{\tau}_h^f,
 \widetilde{\bld v}_h^f)
 \in \Whf{k-1}\times\Shf{k}\times\Vhf{k}\times\tWhf{k}
 \times \tVhf{k}
 $ with $\widetilde{\bld v}_h^f|_{\Gamma_h^f}=0$, 
 where we write 
 $\intV{\eta}{\xi}{f}:=\sum_{K^t\in\Th^{f,t}}\int_{K^t}\eta\cdot\xi\mathrm{dx}$
 as the volume integral, and 
 $\intF{\eta}{\xi}{f}:=\sum_{K^t\in\Th^{f,t}}\int_{\partial K^t}\eta\cdot\xi\mathrm{ds}$
 as the element-boundary integral. 
 Here 
 the {\it viscous} and {\it convective}  numerical fluxes are defined as
 follows: 
 \begin{align}
   \label{flux}
   \widetilde{\bld{Flux}}_{\mathrm v}:= &\;
                            \widetilde{\sigma}_h^f\bld n
                            +2\mu^f\tng( \bld\epsilon_h^f\bld n)
                            -\alpha^f\;\tng(\bld u_h^f-\widetilde{\bld
                            u}_h^f),\\
   \widetilde{\bld{Flux}}_{\mathrm c}:= &\;
 \rho^f(\bld u_h^f-\bld \omega^f)\cdot \bld n \left((\bld u_h^f\cdot\bld n) \bld n
 +\tng(\check{\bld u}_h^{f,up})\right),
 \end{align}
\end{subequations}
where $\tng(\bld u_h^{f,up})$ is the following upwinding flux 
in the tangential direction:
\[
  \tng(\check{\bld u}_h^{f,up}):=\left\{
    \begin{tabular}{ll}
     $ \tng(\bld u_h^f)$ & if $(\bld u_h^f-\bld w^f)\cdot\bld n >0$,\\[3ex]
     $ \tng(\widetilde{\bld u}_h^f)$ & if $(\bld u_h^f-\bld w^f)\cdot\bld n \le 0$,
\end{tabular}
  \right.
\] 
and the (positive) stabilization parameter $\alpha^f$ is taken to be 
$\alpha^f=2\mu^f$.

The following result shows that the fluid velocity approximation is 
globally divergence free.
\begin{lemma}\label{lma:div-free}
  The semi-discrete scheme \eqref{hdg-fluid} produces a 
  globally divergence-free velocity approximation which has a 
  vanishing normal component on the domain boundary, i.e.,
  $
    \bld u_h^f\in H_0(\mathrm{div};\Omega_t^f)$, and 
  $  \mathrm{div}_{\mathbf x}\,\bld u_h^f=0.$
\end{lemma}
\begin{proof}
%
  To simplify notation, we suppress the superscript $t$ in the following derivation.
Since functions in the finite element space $\Vhf{k}$ are transformed via the
Piola mapping, we have
  $\bld u_h^f|_{K}=\frac{1}{J_K}F_K\left(
  \widehat{\bld u}\circ(\Phi_K)^{-1}\right)$
  for some function $\widehat{\bld u}\in [\pol^k(\widehat K)]^d$.
  It is well-known \cite{BoffiBrezziFortin13} that 
  the following property holds for Poila transformations:
\begin{align*}
    \mathrm{div}_{\mathbf x}(\bld u_h^f)\circ\Phi_K
    =&\; \frac{1}{J_K} \mathrm{div}_{\widehat {\mathbf x}}(\widehat{\bld
    u}),
\end{align*}
which implies that
\begin{align*}
  \int_{K} \mathrm{div}_{\mathbf x}(\bld u_h^f)\,
  \widehat q\circ(\Phi_K)^{-1}\mathrm{dx}
  =&\;
  \int_{\widehat K} \mathrm{div}_{\widehat{\mathbf x}}(\widehat{\bld u})
  \widehat q\;\mathrm{d\widehat x}, 
\end{align*}
for any function $\widehat q$ on the reference element $\widehat
K$.
  Since we have 
  $\mathrm{div}_{\widehat{\mathbf x}}(\widehat{\bld u})\in \pol^{k-1}(\widehat
  K)$, by the definition of the scalar finite element space $\Whf{k-1}$,  
  we can take the test function $q_h^f\in \Whf{k-1}$ in \eqref{hdg-f3} such that  
  $q_h^f|_{K} =\mathrm{div}_{\widehat{\mathbf
  x}}(\widehat{\bld u})\circ(\Phi_K)^{-1}$ and 
  $q_h^f=0$ elsewhere, which leads to 
  \[
    0=\int_{K}
    \mathrm{div}_{\mathbf x}(\bld u_h^f)
\,\mathrm{div}_{\widehat{\mathbf x}}(\widehat{\bld u})\circ(\Phi_K)^{-1}
\,\mathrm{dx}
    = 
    \int_{\widehat K}
    (\mathrm{div}_{\widehat{\mathbf x}}(\widehat{\bld
    u}))^2\,\mathrm{d\widehat{x}}.
  \]
Hence,
\begin{align}
  \label{div-free}
    \mathrm{div}_{\mathbf x}(\bld
    u_h^f)|_{K}=
    \frac{1}{J_K}
\mathrm{div}_{\widehat{\mathbf x}}(\widehat{\bld u})\, \circ(\Phi_K)^{-1}
    =0.
\end{align}

    Next, let us prove normal continuity of $\bld u_h^f$ across
    interior element boundaries.
    Let $E_{ij}\in\Eh^{f,t}$ be an interior facet which is the 
    $l$-th facet, $E_{K_i,l}$, of
    element $K_i$,  and the $m$-th facet, $E_{K_j,m}$, of element $K_j$.
    Restricting the equation \eqref{hdg-f4} to the facet 
    $E_{ij}$, we have 
\begin{align}
  \label{normal1}
  \int_{E_{K_i,l}} (\bld u_h^f\cdot\bld n_{K_i,l})
  \widehat{\tau}\circ
  (\Psi_{E_{ij}})^{-1}\mathrm{ds}
+ \int_{E_{K_j,m}} (\bld u_h^f\cdot\bld n_{K_j,m})
\widehat{\tau}\circ(\Psi_{E_{ij}})^{-1} \mathrm{ds} = 0,
\quad \forall \widehat \tau \in\pol^k(\widehat E). 
\end{align}
Let $\widehat {\bld u}_i$ and $\widehat {\bld u}_j$ be functions in 
$[\pol^k(\widehat K)]^d$ such that 
\[
  \bld u_h^f|_{K_i} = 
\frac{1}{J_{K_i}}F_{K_i}\widehat {\bld u}_i, \quad 
  \bld u_h^f|_{K_j} = 
\frac{1}{J_{K_j}}F_{K_j}\widehat {\bld u}_j.
\] 
A simple calculation yields 
\begin{align}
  \label{ratio}
  (\bld u_h^f\cdot \bld n_{K_i,l})\circ\Phi_{K_i}
      =\; 
      \frac{1}{J_{K_i}\|(F_{K_i})^{-T}\widehat{\bld n}_{l}\|} 
      \widehat{\bld u}_i\cdot\widehat {\bld n}_l\quad \text{ on }
      \widehat E_{l},
\end{align}
where $J_{K_i}\|(F_{K_i})^{-T}\widehat{\bld n}_{l}\| = 
|E_{ij}|/|\widehat E_l|
$ is the ratio of facet measures.
Hence, we have 
\begin{align*}
  \int_{E_{K_i,l}} (\bld u_h^f\cdot\bld n_{K_i,l})
  \widehat{\tau}\circ
  (\Psi_{E_{ij}})^{-1}\mathrm{ds}
  = &\; \int_{\widehat E_l}(\widehat{\bld u}_i\cdot\widehat {\bld n}_l)
  \widehat{\tau}\circ(\widehat\Psi_l)^{-1}\,\mathrm{d\widehat s}
  = \; \int_{\widehat E}\left(\widehat J_l (\widehat{\bld u}_i\cdot\widehat {\bld n}_l)
  \circ\widehat\Psi_l\right)
  \widehat{\tau}\,\mathrm{d\widehat s}.
\end{align*}
Substituting the above equation back to \eqref{normal1}, we
get
\[
  \int_{\widehat E}\left(
    \widehat J_l (\widehat{\bld u}_i\cdot\widehat {\bld n}_l)
  \circ\widehat\Psi_l
+
    \widehat J_m (\widehat{\bld u}_j\cdot\widehat {\bld n}_m)
  \circ\widehat\Psi_m
\right)
  \widehat{\tau}\,\mathrm{d\widehat s} = 0, \quad \forall \widehat
  \tau\in\pol^k(\widehat E).
\] 
Since 
$\widehat J_l (\widehat{\bld u}_i\cdot\widehat {\bld n}_l)
  \circ\widehat\Psi_l
+
    \widehat J_m (\widehat{\bld u}_j\cdot\widehat {\bld n}_m)
    \circ\widehat\Psi_m\in \pol^k(\widehat E)$, it must be {\it zero}. 
    Finally, by equation \eqref{ratio} and the fact that 
    $\widehat J_l = |\widehat E_l|/|\widehat E|$, we have
\begin{align}
  \label{normal-cont}
  \left( (\bld u_h^f\cdot \bld n_{K_i,l})|_{K_{i}}+
  (\bld u_h^f\cdot \bld n_{K_i,l}
)|_{K_j}\right)\circ\Psi_{E_{ij}}
  = 
  \frac{|\widehat E|}{|\widehat E_{ij}|}(\widehat{J}_l 
(\widehat{\bld u}_i\cdot\widehat {\bld n}_l)\circ \widehat\Psi_l
+\widehat{J}_m
(\widehat{\bld u}_j\cdot\widehat {\bld n}_m)\circ \widehat\Psi_m)
=0
\end{align}
Similarly, we can prove $\bld u_h^f\cdot\bld n|_{\Gamma_h^f}=0$ by 
working on the equation \eqref{hdg-f4} on the domain boundary.
Combining the results \eqref{div-free} and \eqref{normal-cont}, we readily get
$\bld u_h^f\in H_0(\mathrm{div};\Omega)$ and 
$\mathrm{div}_{\mathbf x}\bld u_h^f=0
$.
This completes the proof.
\end{proof}

To further simplify notation, 
we introduce the following operators
associated with the scheme \eqref{hdg-fluid}:
\begin{subequations}
\label{operators-f}
\small
\begin{align}
  \mathcal{M}_h^f(\bld u_h^f, \bld v_h^f) :=&\; 
  \intV{
  \rho^f
  \bld u_h^f
  }{\bld v_h^f}{f},\\
  \mathcal{C}_h^f\left(\bld u_h^f-\bld \omega^f; (\bld u_h^f,\widetilde{\bld u}_h^f),
  (\bld v_h^f,\widetilde{\bld v}_h^f)\right)
  := &\;
  \intV{
   \rho^f(\mathrm{div}_{\mathbf x}\bld \omega^f)\bld u_h^f
  }{\bld v_h^f}{f}
     -\intV{
    \rho^f(\bld u_h^f-\bld \omega^f)\otimes
    \bld u_h^f
  }{\nabla_{\mathbf x}\bld v_h^f}{f}
  \\\nonumber
     &\;
     \hspace{-.13\textwidth}
  +\intF{\widetilde{\bld {Flux}}_{\mathrm c}}{
  \tng(\bld v_h^f-\widetilde{\bld v}_h^f)}{f}
  +\intF{\rho^f(\bld u_h^f-\bld \omega^f)\cdot\bld n (\bld u_h^f\cdot \bld n)}{
\bld v_h^f\cdot \bld n
}{f},\\
  \mathcal{A}_h^f\left(
    (p_h^f, \bld \epsilon_h^f, \bld u_h^f,
    \widetilde{\sigma}_h^f, \widetilde{\bld u}_h^f
    ),
  (q_h^f,\bld G_h^f, \bld v_h^f,\right.&\left.
    \widetilde{\tau}_h^f, \widetilde{\bld v}_h^f
    )\right)
  := \;
  2\mu^f\intV{\bld \epsilon_h^f
  }{\nabla_{\mathbf x}\bld v_h^f}{f}
  -\intV{p_h^f
  }{\mathrm{div}_{\mathbf x}\bld v_h^f}{f}
 \\\nonumber
 &\;\hspace{-0.1\textwidth}
 -\intF{\widetilde{\bld {Flux}}_{\mathrm v}}{\;\tng(\bld v_h^f-\widetilde{\bld
  v}_h^f)}{f}
  -\intF{\widetilde{\sigma}_h^f}{\bld v_h^f\cdot\bld n}{f}
  \\\nonumber
 &\;\hspace{-0.1\textwidth}
 +2\mu^f \intV{\bld \epsilon_h^f-\mathbf D_{\mathbf x}(\bld u_h^f)}{\bld G_h^f}{f}
  +2\mu^f\intF{\tng(\bld u_h^f-\widetilde{\bld u}_h^f)}{\bld G_h^f\bld n}{f}
 \\\nonumber
 &\;\hspace{-0.1\textwidth}
  +\intV{\mathrm{div}_{\mathbf x}\bld u_h^f}{q_h^f}{f} 
  +\intF{\bld u_h^f\cdot\bld n}{\widetilde\tau_h^f}{f},\\
  \mathcal{F}_h^f(\bld v_h^f):=&\;
  \intV{\rho^f\bld f^f}{\bld v_h^f}{f}.
\end{align}
\end{subequations}
Then the equations in \eqref{hdg-fluid} can be expressed as the following
compact form:
\begin{align}
  \label{hdg-fluid-c}
  \mathcal{M}_h^f(
  \left.\frac{\partial \bld u_h^f}{\partial t}\right|_{\mathbf x_0} 
  , \bld v_h^f) 
  +
  \mathcal{C}_h^f\left(\bld u_h^f-\bld \omega^f; (\bld u_h^f,\widetilde{\bld u}_h^f),
  (\bld v_h^f,\widetilde{\bld v}_h^f)\right)\quad\quad& \\\nonumber
  +\mathcal{A}_h^f\left(
    (p_h^f, \bld \epsilon_h^f, \bld u_h^f,
    \widetilde{\sigma}_h^f, \widetilde{\bld u}_h^f
    ),
  (q_h^f,\bld G_h^f, \bld v_h^f,
    \widetilde{\tau}_h^f, \widetilde{\bld v}_h^f
)\right)&=\; \mathcal{F}_h^f(\bld v_h^f).
\end{align}

Energy stability of the semi-discrete scheme is given below.
\begin{theorem}\label{thm:stability}
Let 
$
(p_h^f,
 \bld \epsilon_h^f,
 \bld u_h^f,
 \widetilde{\sigma}_h^f,
 \widetilde{\bld u}_h^f)
 \in \Whf{k-1}\times\Shf{k}\times\Vhf{k}\times\tWhf{k}
 \times \tVhf{k}
 $ be the solution to the scheme \eqref{hdg-fluid}.
 The following energy identity holds:
\[
\frac12 \frac{d}{dt}\intV{
  \rho^f\bld u_h^f}{\bld u_h^f}{f}
+
2\mu^f\intV{\bld \epsilon_h^f}{\bld \epsilon_h^f}{f}
+\intF{\gamma_h\left|\tng(\bld u_h^f-\widetilde{\bld
    u}_h^f)\right|^2}{1}{f}
  =\intV{\rho^f\bld f^f}{\bld v_h^f}{f},
\] 
where the parameter 
$
  \gamma_h := \alpha^f + \rho^f|(\bld u_h^f-\bld \omega_h^f)\cdot\bld n|
$
is positive.
\end{theorem}
\begin{proof}
Taking test functions in equations \eqref{hdg-fluid} 
as the trial functions, we get
\begin{align}
  \label{X1}
  T_1+T_2+T_3 = 
  \mathcal{F}_h^f(\bld v_h^f),
\end{align}
where
\begin{align*}
  T_{1} = 
  &\;
  \mathcal{M}_h^f(
  \left.\frac{\partial \bld u_h^f}{\partial t}\right|_{\mathbf x_0} 
  , \bld u_h^f),\quad \quad T_{2} =\;
  \mathcal{C}_h^f\left(\bld u_h^f-\bld \omega^f; (\bld u_h^f,\widetilde{\bld u}_h^f),
  (\bld u_h^f,\widetilde{\bld u}_h^f)\right),& \\\nonumber
  T_{3} =&\;
  \mathcal{A}_h^f\left(
    (p_h^f, \bld \epsilon_h^f, \bld u_h^f,
    \widetilde{\sigma}_h^f, \widetilde{\bld u}_h^f
    ),
  (p_h^f,\bld \epsilon_h^f, \bld u_h^f,
    \widetilde{\sigma}_h^f, \widetilde{\bld u}_h^f
)\right). 
\end{align*}

Using the well-known identity 
$
  \frac{\partial }{\partial t}J_0^t = J_0^t\,\mathrm{div}_{\mathbf x}\bld
  \omega^f,
$
 we get 
\begin{align*}
  T_{1} = 
  &\; 
\frac12 \frac{d}{dt}\intVo{
  \rho^f\bld u_h^f}{\bld u_h^f}{f}
-
\frac12 \intVo{
(\mathrm{div}_{\mathbf x}\bld \omega^f)\rho^f\bld u_h^f}{\bld u_h^f}{f}\\
=&\;
\frac12 \frac{d}{dt}\intV{
  \rho^f\bld u_h^f}{\bld u_h^f}{f}
-
\frac12 \intV{\rho^f
(\mathrm{div}_{\mathbf x}\bld \omega^f)\bld u_h^f}{\bld u_h^f}{f},
\end{align*}

Integration by parts 
yields that 
\begin{align*}
  \intV{
    \rho^f(\bld u_h^f-\bld \omega^f)\otimes
    \bld u_h^f
  }{\nabla_{\mathbf x}\bld u_h^f}{f}
  =&\;
 -\frac12 \intV{
   \rho^f\mathrm{div}(\bld u_h^f-\bld \omega^f)
    \bld u_h^f
  }{\bld u_h^f}{f}
  +\frac12\intF{
    \rho^f(\bld u_h^f-\bld \omega^f)\cdot\bld n \;
    \bld u_h^f
  }{\bld u_h^f}{f}\\
  &\hspace{-2cm} =\;
 \frac12 \intV{
   \rho^f(\mathrm{div}\bld \omega^f)
    \bld u_h^f
  }{\bld u_h^f}{f}
  +\frac12\intF{
    \rho^f(\bld u_h^f-\bld \omega^f)\cdot\bld n \;
    \tng(\bld u_h^f)
  }{\tng(\bld u_h^f)}{f},
\end{align*}
where we 
used the fact that 
  $\bld u_h^f\in H_0(\mathrm{div};\Omega_f^t)$ and 
  $\mathrm{div}_{\mathbf x}\bld u_h^f=0$ 
in the last step.
Combining the above identity with the definition of $T_2$, 
and the definition 
of the upwinding flux $\check {\bld u}_h^{f,up}$
and simplifying the terms, we get
\begin{align*}
  T_2 = 
 \frac12 \intV{
   \rho^f(\mathrm{div}\bld \omega^f)
    \bld u_h^f
  }{\bld u_h^f}{f}
  +\frac12\intF{
    \rho^f\left|(\bld u_h^f-\bld \omega^f)\cdot\bld n\right| \;
    \tng( \bld u_h^f-\widetilde{\bld u}_h^f)
  }{
    \tng( \bld u_h^f-\widetilde{\bld u}_h^f)
}{f}
\end{align*}

Simplifying the terms in $T_3$, we obtain
\[
T_3 = 
2\mu^f\intV{\bld \epsilon_h^f}{\bld \epsilon_h^f}{f}
+\intF{\alpha^f\,\tng(\bld u_h^f-\widetilde{\bld
    u}_h^f)}{\tng(\bld u_h^f-\widetilde {\bld u}_h^f)}{f}.
\] 
The equality in Theorem \ref{thm:stability} is then obtained by simply 
substituting the above three terms back to equation \eqref{X1}.
\end{proof}

\subsection{ALE-divergence-free HDG scheme: temporal discretization}
In this subsection,
we consider the temporal discretization 
for the semidiscrete scheme
\eqref{hdg-fluid}. 

We work on the time derivative term in \eqref{hdg-f1}, restricting to 
a single element $K^t\in\Th^{f,t}$.
Let 
$
  \left\{\widehat{\bld \psi}_i(\widehat{\mathbf x})
 \right \}_{i=1}^N
$
be a set of basis functions for the space $[\pol^k(\widehat K)]^d$,
where 
$
  N = d{{n+d}\choose{k}}=
\mathrm{dim}[\pol^k(\widehat K)]^d.
$ 
Then the function $\bld u_h^f$, restricting to the element $K^t$, 
can be expressed as 
\[
  \left.  \bld u_h^f\right|_{K^t} = 
    \frac{1}{J_K^t}F_K^t\sum_{i=1}^N
    \mathsf{u}_i(t)
    \widehat{\bld \psi}_i \circ(\Phi_K^t)^{-1}
    = 
    \frac{1}{J_0^t}F_0^t\sum_{i=1}^N
    \mathsf{u}_i(t)
    {\bld \psi}_i^0 \circ(\bld \phi_t)^{-1},
\] 
where $\mathsf{u}_i:[0,T]\rightarrow \mathbb{R}$ 
is the coefficient associated to the $i$-th basis, and 
\[
  {\bld \psi}_i^0(\mathbf {x}_0) = 
    \frac{1}{J_K^0}F_K^0
    \widehat{\bld \psi}_i \circ(\Phi_K^0)^{-1}(\mathbf x_0)
\]
is the mapped (time-independent) basis on the initial element $K^0$.
Applying the chain rule, we get
\begin{align}
  \label{ale-time}
  \left.\frac{\partial \bld u_h^f}{\partial t}\right|_{\mathbf x_0}
    = &\;
  \frac{\partial }{\partial t}(\frac{1}{J_0^t}F_0^t)
  \sum_{i=1}^N
    \mathsf{u}_i(t)
    {\bld \psi}_i^0 
  \circ(\bld \phi_t)^{-1}
    +
(\frac{1}{J_0^t}F_0^t)\sum_{i=1}^N
  \frac{d}{dt}\mathsf{u}_i(t)
    {\bld \psi}_i^0
  \circ(\bld \phi_t)^{-1}
    \\
  \nonumber
    = &\;
\left(
  \nabla_{\mathbf x}\bld \omega^f
  -(\mathrm{div}_{\mathbf x}\bld \omega^f)\mathbf I
\right)
\frac{1}{J_0^t}F_0^t
  \sum_{i=1}^N
    \mathsf{u}_i(t)
    {\bld \psi}_i^0 
  \circ(\bld \phi_t)^{-1}
    +
(\frac{1}{J_0^t}F_0^t)\sum_{i=1}^N
  \frac{d}{dt}\mathsf{u}_i(t)
    {\bld \psi}_i^0
  \circ(\bld \phi_t)^{-1}
\end{align}
where the first term in the above right hand side is due to the use 
of the (time-dependent) Piola mapping; see \cite{Neunteufel21} for a similar derivation.
Now time derivative only appears in the coefficients $\mathsf{u}_i$ in the above
right hand side, which can be discretized using a standard ODE solver. 
We use the second-order backward difference formula (BDF2) in this work.

To simplify notation, we denote 
$\bld u_h^{f,n}$ as the solution $\bld u_h^f$ with 
coefficients $\mathsf{u}_i$ evaluated at time 
$t^n=n\Delta t$, where $\Delta t$ is the time step size, i.e., 
\[
  \left.  \bld u_h^{f,n}\right|_{K^t} = 
    \frac{1}{J_0^t}F_0^t\sum_{i=1}^N
    \mathsf{u}_i(t^n)
    {\bld \psi}_i^0 \circ(\bld \phi_t)^{-1}, \quad \forall t\in[0,T].
\] 
Then, the BDF2 time discretization of the term in \eqref{ale-time} at time $t^n$
is 
\begin{align}
  \label{ale-bdf2}
    D_t^2\bld u_h^{f,n}:=
\frac{
  1.5 \bld u_h^{f,n}-2\bld u_h^{f,n-1}+0.5\bld u_h^{f,n-2}
}{\Delta t}
+\left(
  \nabla_{\mathbf x}\bld \omega^{f,n}
  -(\mathrm{div}_{\mathbf x}\bld \omega^{f,n})\mathbf I
\right)\bld u_h^{f,n},
  \end{align}
  where $\bld \omega^{f,n}=\bld\omega^f(t^n)$ is the mesh velocity at time 
  $t^n$.

 The fully discrete ALE-divergence-free HDG scheme with BDF2 time stepping is
 obtained from the semidiscrete scheme \eqref{hdg-fluid} by 
replacing the term $
\left.\frac{\partial \bld u_h^f}{\partial t}(t^n)\right|_{\mathbf x_0}
  $ with $D_t^2\bld u_h^{f,n}$ in \eqref{ale-bdf2}, and
 evaluating all the other terms 
 at
 time level $t^n$:
 Given $\bld u_h^{f,n-2}, \bld u_h^{f,n-1}\in\Vhf{k}$ at time $t^{n-2}$ and 
 $t^{n-1}$, find
$
(p_h^{f,n},
\bld \epsilon_h^{f,n},
\bld u_h^{f,n},
\widetilde{\sigma}_h^{f,n},
\widetilde{\bld u}_h^{f,n})
 \in \Whf{k-1}\times\Shf{k}\times\Vhf{k}\times\tWhf{k}
 \times \tVhf{k}
 $ 
at time $t^n$
with $\widetilde{\bld u}_h^{f,n}|_{\Gamma_h^f}=0$
such that 
\begin{align}
  \label{hdg-fluidF}
  \mathcal{M}_h^f\left(
    D_t^2\bld u_h^{f,n}, \bld v_h^f\right) 
  +
  \mathcal{C}_h^f\left(\bld u_h^{f,n}-\bld \omega^{f,n}; (\bld
    u_h^{f,n},\widetilde{\bld u}_h^{f,n}),
  (\bld v_h^f,\widetilde{\bld v}_h^f)\right)\quad\quad&\\\nonumber
  +\mathcal{A}_h^f\left(
    (p_h^{f,n}, \bld \epsilon_h^{f,n}, \bld u_h^{f,n},
    \widetilde{\sigma}_h^{f,n}, \widetilde{\bld u}_h^{f,n}
    ),
  (q_h^f,\bld G_h^f, \bld v_h^f,
    \widetilde{\tau}_h^f, \widetilde{\bld v}_h^f
)\right)&=\; \mathcal{F}_h^f(\bld v_h^f).
\end{align}
  for all 
$
(q_h^f,
 \bld G_h^f,
 \bld v_h^f,
 \widetilde{\tau}_h^f,
 \widetilde{\bld v}_h^f)
 \in \Whf{k-1}\times\Shf{k}\times\Vhf{k}\times\tWhf{k}
 \times \tVhf{k}
 $ with $\widetilde{\bld v}_h^f|_{\Gamma_h^f}=0$.

\begin{remark}[Comment on fully discrete stability]\label{stability}
Unlike the static mesh case where a fully  implicit BDF2 scheme leads to
unconditional energy stability, the ALE moving mesh scheme \eqref{hdg-fluidF} 
can be only shown to be
{\it conditionally stable} with maximum allowable time step size $\Delta t$ depending
on the mesh velocity $\bld \omega^f$. We omit the proof of this result 
and refer interested reader to the work \cite{Nobile01} for a similar derivation.
\end{remark}

\begin{remark}[Semi-implicit convection treatment]
  The scheme \eqref{hdg-fluidF} leads to a nonlinear system due to 
  a fully implicit treatment of the nonlinear convection term.
  A slightly cheaper method is to treat the convection term semi-implicitly
  by replacing the convection velocity term $\bld u_h^{f,n}-\bld \omega^{f,n}$
  in equation \eqref{hdg-fluidF}  
  with 
 the following second-order extrapolation
 \[
   (2\bld u_h^{f,n-1}-\bld u_h^{f,n-2})-\bld \omega^{f,n}.
 \]
 This leads to a linear scheme 
 with a similar stability property as the original nonlinear scheme \eqref{hdg-fluidF}. 
\end{remark}

\section{The $H(\mathrm{curl})$-conforming HDG scheme for nonlinear elasticity}
\label{sec:elas}
In this section, we introduce the 
$H(\mathrm{curl})$-conforming HDG scheme 
for nonlinear elasticity in the Lagrangian framework.

\subsection{The equations of elastodynamics}
We consider the following nonlinear elasticity problem 
with a hyperelastic material in a 
Hu-Washizu formulation \cite{Washizu75} on the fixed 
reference domain $\Omega^s\in\mathbb{R}$:
\begin{subequations}
\label{elas-eq}
  \begin{alignat}{2}
\label{elas-eq0}
 \frac{\partial }{\partial t}\bld d
-\bld u^s
    =&\; 0, \quad&& \text{in} \;\Omega^s\times [0,T]\\
\label{elas-eq1}
\rho^s \frac{\partial }{\partial t}\bld u^s
-\mathrm{div}\bld P
    =&\; \rho^s\bld f^s, \quad&& \text{in} \;\Omega^s\times [0,T]\\
\label{elas-eq2}
\bld P - \frac{\partial \Psi(\bld F)}{\partial \bld F}
    =&\; 0, \quad&& \text{in} \;\Omega^s\times [0,T]\\
\label{elas-eq3}
\bld F - \nabla\bld d -\mathbf I
    =&\; 0, \quad&& \text{in} \;\Omega^s\times [0,T]
\end{alignat}
\end{subequations}
where 
$\bld d$ is the structure displacement field,
$\bld u^s$ is the structure velocity (on reference domain),
$\bld F$ is the deformation gradient, 
$\bld P$ is 
the first Piola-Kirchhoff stress tensor, and
$\Psi(\bld F)$ is the hyperelastic potential, where we use the following 
Saint Venant-Kirchhoff model in this work
\[
  \Psi(\bld F):= \frac{\lambda^s}{2}\mathrm{tr}(\bld E)^2 +\mu^s\bld E:\bld E,
  \quad\quad \bld E:=\frac12(\bld F^T\bld F-\mathbf I),
\] 
with $\lambda^s, \mu^s$ being the two Lam\'e parameters.
In this case, we have
\[
  \frac{\partial \Psi}{\partial \bld F} = \bld F(\lambda^s \mathrm{tr}(\bld
  E)\mathbf I
  +2\mu^s\bld E).
\]


\subsection{Mesh and finite element spaces}
\subsubsection{Mesh and mappings}
We consider a similar mesh setting as that in Section \ref{sec:fluid}, except
that the structure domain and the associated mesh
do not move over time.
In particular, $\Th^{s}:=\{K\}$ is a conforming simplicial
triangulation of the structure
 domain $\Omega^s\subset \mathbb{R}^d$, 
where the element $K=\Phi_K(\widehat K)$ is a mapped
simplex from the reference simplex element 
$\widehat K$,
$\partial \Th^{s}:=\{\partial K:=\Phi_K(\partial \widehat K)\}$ is  the collection of element
boundaries of the mesh $\Th^{s}$, and 
$\Eh^{s}:=\{E:=\Psi_E(\widehat E) \}$ is the mesh skeleton, 
where $\Psi_E:\widehat E\rightarrow E$ 
is the mapping between the reference surface element $\widehat E$ and the 
physical facet $E$.

\subsubsection{The finite element spaces}
Again, we introduce both finite element spaces on the structure mesh $\Th^s$, and 
the mesh skeleton $\Eh^s$:
\begin{subequations}
  \label{spaces-S}
\begin{align}
  \bld \Sigma_h^{k, s}:=
                     &\;\{\bld \sigma \in [L^2(\Omega^s)]^{d\times d}: \;\;
    \bld \sigma|_{K} = \widehat{\bld \sigma}\circ (\Phi_K)^{-1},\;\; 
    \widehat{\bld \sigma}\in [\pol^k(\widehat K)]^{d\times d}, 
\;\forall K\in \Th^{s}\},\\ 
  \bld V_h^{k,s}:=&\;\left\{\bld v\in H(\mathrm{curl};\Omega^s): \;\;
    \bld v|_{K} = F_K^{-T} \left(\widehat {\bld
      v}\circ(\Phi_K)^{-1}\right), 
      \;\; \widehat{\bld v}\in [\mathcal{P}^k(\widehat K)]^d, \;\forall
    K\in \Th^{s}\right\}, \\
        \widetilde{\bld V}_h^{k,s}:=&\;\{\widetilde{\bld v}
          \in [L^2(\Eh^{s})]^d: \;\;\widetilde{\bld v}|_{E}
          = 
          \mathrm{nrm}\left(
          \frac{1}{J_E}F_E \,\widehat {\bld
        v}\circ(\Psi_E)^{-1}\right),
      \;\;\widehat {\bld v}\in
        [\pol^k(\widehat E)]^d,\quad \forall E\in\Eh^{s}\},
    \end{align}
where $\nrm(\bld v)|_{E} := (\bld v\cdot \bld n_E)\bld n_E$
denotes the normal component of the vector $\bld v$ on the facet $E$, whose
normal direction is $\bld n_E$.
    Note that 
the standard pull-back mapping is used to define the
    the non-symmetric and discontinuous tensor space $\Shs{k}$,
    which will be used to approximate the tensor fields $\bld F/\bld P$, 
    the covariant mapping is used to define the 
$H(\mathrm{curl})$-conforming 
    vector space $\Vhs{k}$, which 
    will be used to approximate  the structure displacement and velocity,
    and the Piola mapping is used to define the hybrid space
    $\tVhs{k}$, which 
    will be used to approximate the normal component of 
    structure displacement and velocity on the mesh skeleton.
  \end{subequations}

\subsection{The $H(\mathrm{curl})$-conforming HDG scheme: spatial discretization}
In this subsection, we focus on the spatial discretization of the 
equations \eqref{elas-eq} with 
homogeneous Dirichlet boundary conditions 
\begin{align}
  \label{dirichletS}
  \bld d(\mathbf x,t)= \bld u^s(\mathbf x,t) = 0, \quad  \forall \mathbf x\in \partial\Omega^s, \quad t\in[0,T].
\end{align}

We use polynomials of degree $k\ge 1$ for all the variables.
The spatial HDG discretization reads as follows: 
Find $
(\bld P_h,
 \bld F_h,
 \bld d_h,
 \bld u_h^s,
 \widetilde{\bld d}_h,
 ,\widetilde{\bld u}_h^s)
 \in \Shs{k}\times\Shs{k}\times\Vhs{k}\times\Vhs{k}
 \times \tVhs{k}\times \tVhs{k}
 $ with 
 $
   \tng({\bld d}_h)|_{\partial\Omega^s} =
     \nrm(\widetilde{\bld
   d}_h)|_{\partial\Omega^s}=0,
 $ 
such that 
\begin{subequations}
  \label{hdg-elas}
  \begin{align}
    \label{elas-vel}
    \intVS{\frac{\partial \bld d_h}{\partial t}-\bld u_h^s}{\bld v_h^s}
   + \intFS{\frac{\partial \widetilde{\bld d}_h}{\partial t}-
   \widetilde{\bld u}_h^s}{\widetilde{\bld v}_h^s}&=0,\\
  \label{hdg-s1}
  \intVS{
  \rho^s\frac{\partial \bld u_h^s}{\partial t}
  }{\bld \xi_h}
  +\intVS{\bld P_h}{\nabla\bld \xi_h}
  -\intFS{\widetilde{\bld P_h}\bld n}{\;\nrm(\bld \xi_h)}&\;=\;
  \intVS{\rho^s\bld f^s}{\bld \xi}, \\
  \label{hdg-s2}
  \intVS{
    \frac{\partial \Psi(\bld F_h)}{\partial \bld F} -   \bld P_h
  }{\bld G_h}
                                                  &\;=\;0,\\
  \label{hdg-s3}
  \intVS{
    \bld F_h - \nabla \bld d_h-\mathbf I
  }{\bld Q_h}
  +\intFS{\;\nrm(\bld d_h-\widetilde{\bld d}_h)}{{\bld Q_h}\bld n}&\;=\;0,\\
  \label{hdg-s4}
  \intFS{\widetilde{\bld{P}_h}\bld n}{\;\nrm(\widetilde{\bld \xi}_h)} &\;=\;0,
  \end{align}
  for all 
$
(\bld Q_h,\bld G_h, \bld \xi_h, \bld v_h^s, 
\widetilde{\bld \xi}_h, \widetilde {\bld v}_h^s
 )
 \in \Shs{k}\times\Shs{k}\times\Vhs{k}\times\Vhs{k}\times\tVhs{k} \times\tVhs{k}
 $ with 
 $
   \tng({\bld \xi}_h)|_{\partial\Omega^s}
 =\nrm(\widetilde{\bld
   \xi}_h)|_{\partial\Omega^s}=0,
 $
 where
 the numerical flux is 
 \begin{align}
   \label{fluxS}
   \widetilde{\bld{P}_h}\bld n:= &\;
   \bld P_h\bld n - 
   \alpha_s
   \,\nrm(\bld d_h-\widetilde{\bld d}_h),
 \end{align}
with the stabilization parameter taken to be $\alpha_s=2\mu^s$.
\end{subequations}
Note that equation \eqref{elas-vel} implies that 
\begin{align}
  \label{vel-disp}
  \frac{\partial\bld d_h}{\partial t}=\bld u_h^s, \text{ and }
  \frac{\partial\widetilde{\bld d}_h}{\partial t}=\widetilde{\bld u}_h^s.
\end{align}

To further simplify notation, 
we introduce the following operators:
\begin{subequations}
\label{operators-s}
\small
\begin{align}
  \mathcal{M}_h^s(\bld u_h^s, \bld \xi_h) :=&\;
  \intVS{
  \rho^s
  \bld u_h^s
  }{\bld \xi_h},\\
  \mathcal{A}_h^s\left(
    (\bld P_h, \bld F_h, \bld d_h,
    \widetilde{\bld d}_h),
    (\bld Q_h, \bld G_h, \bld \xi_h,
    \widetilde{\bld \xi}_h)\right)
  :=& \;
  \intVS{\bld P_h}{\nabla\bld \xi_h}
  -\intFS{\widetilde{\bld P_h}\bld n}{\;\nrm(\bld \xi_h
-\widetilde{\bld \xi}_h
)}
\\\nonumber
    & \hspace{-0.25\textwidth}+ \intVS{
      \frac{\partial \Psi(\bld F_h)}{\partial \bld F} -   \bld P_h
  }{\bld G_h}
  +\intVS{
    \bld F_h - \nabla \bld d_h-\mathbf I
  }{\bld Q_h}
  +\intFS{\;\nrm(\bld d_h-\widetilde{\bld d}_h)}{{\bld Q_h}\bld n},\\
  \mathcal{F}_h^s(\bld \xi):=&\;
  \intVS{\rho^s\bld f^s}{\bld\xi_h}.
\end{align}
\end{subequations}
Then the equations \eqref{hdg-s1}--\eqref{hdg-s4} can be expressed as the following
compact form:
\begin{align}
  \label{hdg-elas-c}
  \mathcal{M}_h^s\left(\frac{\partial \bld u_h^s}{\partial t},
  \bld \xi_h\right) 
  +\mathcal{A}_h^s\left(
    (\bld P_h, \bld F_h, \bld d_h,
    \widetilde{\bld d}_h),
    (\bld Q_h, \bld G_h, \bld \xi_h,
    \widetilde{\bld \xi}_h)\right)
&=\; \mathcal{F}_h^s(\bld \xi_h).
\end{align}

We have the following energy stability result.
\begin{theorem}
  \label{thm:elas}
  Let $
(\bld P_h,
 \bld F_h,
 \bld d_h,
 \bld u_h^s,
 \widetilde{\bld d}_h,
 ,\widetilde{\bld u}_h^s)
 \in \Shs{k}\times\Shs{k}\times\Vhs{k}\times\Vhs{k}
 \times \tVhs{k}\times \tVhs{k}
 $ 
  be the solution to the scheme \eqref{hdg-elas}.
 Then the following energy identity holds:
\[
  \frac{d}{dt}\mathsf{E}_h^s
  =\intVS{\rho^f\bld f^s}{\bld u_h^s},
\] 
where the elastic energy 
\begin{align}
  \label{elas-ener}
\mathsf{E}_h^s:=
\frac12  \intVS{
  \rho^s\bld u_h^s}{\bld u_h^s}
  +\intVS{\Psi(\bld F_h)}{1}
  +\frac12\intFS{\alpha^s|\nrm(\bld d_h-\widetilde{\bld d}_h)|^2}{1}
\end{align}
\end{theorem}
\begin{proof}
  Taking the test function $\bld \xi_h=\bld u_h^s$ in \eqref{hdg-s1}
  and $\widetilde{\bld \xi}_h=\widetilde{\bld u}_h^s$ in \eqref{hdg-s4}
  and adding, we get 
  \[
   \frac12 \frac{d}{dt}\mathcal{M}_h^s(\bld u_h^s, \bld u_h^s)
    +
  \intVS{\bld P_h}{\nabla\bld u_h^s}
  -\intFS{\widetilde{\bld P_h}\bld n}{\;\nrm(\bld u_h^s-\widetilde{\bld u}_h^s)}
  \;=\;
  \intVS{\rho^s\bld f^s}{\bld u_h^s}.
  \] 
  Taking the test function $\bld G_h=\frac{\partial \bld F_h}{\partial t}$ in
  \eqref{hdg-s2}, we get
  \[
    \frac{d}{dt}\intVS{\Psi(\bld F_h)}{1}
    -
  \intVS{\bld P_h}{
\frac{\partial \bld F_h}{\partial t}
}=0.
  \]
  Taking time derivative of the equation \eqref{hdg-s3} and then taking 
the test function  $\bld Q_h=\bld P_h$ and using 
the relation \eqref{vel-disp}, we get
  \[
  \intVS
  {\frac{\partial \bld F_h}{\partial t}-\nabla\bld u_h^s}
  {\bld P_h}
  +\intFS{\;\nrm(\bld u_h^s-\widetilde{\bld u}_h^s)}{{\bld P_h}\bld n}=0.
  \]
  Adding the above equations and using the definition of the numerical flux 
  \eqref{fluxS}, we get the 
equality in Theorem \ref{thm:elas}.
\end{proof}
 \begin{remark}[Relation with the TDNNS method]
We specifically mention that the use of an 
$H(\mathrm{curl})$-conforming 
space $\Vhs{k}$ for
the displacement field 
   is inspired by the recent TDNNS method 
 \cite{PechsteinSchoberl11}, where good performance
 for nonlinear elastostatics 
 was obersved in
 \cite{neunteufel2020threefield}
 for several numerical examples
including  nearly incompressible materials and thick/thin structures.
A major difference 
between the TDNNS method \cite{neunteufel2020threefield}
and the current HDG method \eqref{hdg-elas} is on the choice of 
the finite element spaces for the tensor fields $\bld P_h$ and $\bld F_h$.
We use a simple tensor space with standard pull-back mappings from the reference
element for both tensor fields, and use the (positive) HDG stabilization term in the
numerical flux to enforce 
the stability of the scheme.
On the other hand, the TDNNS method \cite{neunteufel2020threefield} use 
a doublly Piola mapping for the symmetric part of the stress $\bld P_h$, and 
a doublly covariant mapping for the symmetric part of the deformation gradient 
$\bld F_h$ where the stabilization parameter can be set to be zero; see more details in \cite{neunteufel2020threefield}.

The TDNNS method was known to be particularly suitable for thin structures 
where anisotripic elements are needed 
\cite{Pechstein12,Pechstein17}.
We expect similar results for the current 
 HDG scheme \eqref{hdg-elas}
due to the use of an $H(\mathrm{curl})$-conforming displacement approximation.
\end{remark}

\subsection{The $H(\mathrm{curl})$-conforming HDG scheme: temporal discretization}
Since the semidiscrete scheme \eqref{hdg-elas} is constructed in the Lagrangian
framework on a fixed domain, a standard method-of-lines approach can be readily
used to discretize the time derivatives.
Here we again use the BDF2 method as an example.

The fully discrete scheme is obtained from \eqref{hdg-elas} by replacing the 
time derivative terms in \eqref{hdg-s1}  
and \eqref{elas-vel} 
by the following BDF2 approximation:
\begin{subequations}
  \label{bdf2-s}
  \begin{align}
  \label{bdf2-s1}
  \mathsf{D}_t^2\bld u_h^{s,n}:=
  &\; \frac{1.5\bld u_h^{s,n}-2\bld u_h^{s,n-1}+0.5\bld u_h^{s,n-2}}{\Delta
  t},\\
  \mathsf{D}_t^2\bld d_h^{n}:=
  &\; \frac{1.5\bld d_h^{n}-2\bld d_h^{n-1}+0.5\bld d_h^{n-2}}{\Delta
  t},\\
  \mathsf{D}_t^2\widetilde{\bld d}_h^{n}:=
  &\; \frac{1.5\widetilde{\bld d}_h^{n}-2\widetilde{\bld d}_h^{n-1}+0.5\widetilde{\bld d}_h^{n-2}}{\Delta
  t},
\end{align}
\end{subequations}
and evaluating all the other terms at time level $t^n$:
 Given 
 $\bld u_h^{s,n-2}, \bld d_h^{n-2}\in\Vhs{k}$ and 
 $\widetilde{\bld d}_h^{n-2}\in\tVhs{k}$ at time $t^{n-2}$, and 
 $\bld u_h^{s,n-1}, \bld d_h^{n-1}\in\Vhs{k}$ and 
 $\widetilde{\bld d}_h^{n-1}\in\tVhs{k}$ at time $t^{n-1}$, 
 find
$
(\bld P_h^{n},
\bld F_h^{n},
\bld u_h^{s,n},
\widetilde{\bld u}_h^{s,n})
 \in \Shs{k}\times\Shs{k}\times\Vhs{k}\times\tVhs{k}
 $ 
at time $t^n$
with $
\tng({\bld u}_h^{s,n})|_{\Gamma_h^f}=
\nrm(\widetilde{\bld u}_h^{s,n})|_{\Gamma_h^f}=0$
such that 
\begin{align}
  \label{hdg-elasF}
  \mathcal{M}_h^s\left(\mathsf{D}_t^2\bld u_h^{s,n},
  \bld \xi_h\right) 
  +\mathcal{A}_h^s\left(
    (\bld P_h^n, \bld F_h^n, \bld d_h^n,
    \widetilde{\bld d}_h^n),
    (\bld Q_h, \bld G_h, \bld \xi_h,
    \widetilde{\bld \xi}_h)\right)
&=\; \mathcal{F}_h^s(\bld \xi_h),
\end{align}
  for all 
$
(\bld Q_h,\bld G_h, \bld \xi_h, \bld v_h^s, 
\widetilde{\bld \xi}_h, \widetilde {\bld v}_h^s
 )
 \in \Shs{k}\times\Shs{k}\times\Vhs{k}\times\Vhs{k}\times\tVhs{k} \times\tVhs{k}
 $ with 
 $
   \tng({\bld \xi}_h)|_{\partial\Omega^s}
 =\nrm(\widetilde{\bld
   \xi}_h)|_{\partial\Omega^s}=0,
 $
where
the displacment fields satisfy
\begin{align}
  \label{disp-X}
  \bld d_h^n=\frac43\bld d_h^{n-1}-\frac13\bld d_h^{n-2}+\Delta t\bld u_h^{s,n},
  \quad
  \widetilde{\bld d}_h^n=\frac43\widetilde{\bld d}_h^{n-1}-
  \frac13\widetilde{\bld d}_h^{n-2}+\Delta t\widetilde{\bld u}_h^{s,n}.
\end{align}
%

\section{
  HDG for FSI: coupling and mesh movement}
\label{sec:ale}
In this section, we introduce an HDG method to solve the FSI problem by 
combing the ALE-divergence-free HDG scheme for fluids in Section \ref{sec:fluid}
and 
the $H(\mathrm{curl})$-conforming HDG scheme for structure in Section
\ref{sec:elas}. We mainly focus on the proper treatment of the fluid-structure
interface, and the construction of the ALE map in the fluid domain.
\subsection{The FSI problem}
We consider the interaction between an incompressible, viscous fluid and
an elastic structure.
We use the same notation as the previous two sections.
In particular, the Navier-Stokes equations \eqref{ns-eq-reform} is considered 
on  the (moving) fluid domain 
$\Omega^f_t$, 
and the equations of nonlinear hypreelasticity \eqref{elas-eq} is considered 
on the reference structure domain $\Omega^s$.
We denote $\Gamma_0 = \Omega_0^f\cap \Omega^s$ as the fluid-structure interface in the
initial configuration, and denote $\Gamma_t=\bld \phi_t(\Gamma_0)$ as the
deformed interface.
The equations \eqref{ns-eq-reform} and \eqref{elas-eq} are equiped with the
following coupling and boundary conditions:
\begin{subequations}
  \label{fsi-bc}
  \begin{alignat}{2}
  \label{fsi-ic1}
  \bld u^f-\bld u^s\circ(\bld\phi_t)^{-1} = &\;0 &&\quad\text{ on } \Gamma_t,\\
  \label{fsi-ic2}
  \bld \sigma^f\bld n^f
  +(\bld P\bld n^s)\circ(\bld\phi_t)^{-1} = &\;0 &&\quad\text{ on } \Gamma_t,\\
  \label{fsi-bc1}
  \bld u^f = &\;0 &&\quad\text{ on } \partial\Omega^f_t\backslash\Gamma_t,\\
  \label{fsi-bc2}
  \bld d = \bld u^s=&\; 0 &&\quad\text{ on } \partial\Omega^s\backslash\Gamma_0,
\end{alignat}
\end{subequations}
where
$\bld n^f$ is the outward normal direction on the deformed interface $\Gamma_t$
from the fluid domain $\Omega_t^f$,
$\bld n^s$ is the outward normal direction on the initial interface $\Gamma_0$
from the structure domain $\Omega^s$, and
the ALE map $\bld \phi_t$ shall satisfy the following boundary conditions 
\begin{alignat}{2}
  \label{fsi-map}
  \bld\phi_t(\mathbf x_0) = &\;\mathbf x_0 + \bld d(\mathbf x_0),&&\quad 
  \forall \mathbf x_0\in\Gamma_0,\quad\quad 
  \bld\phi_t(\mathbf x_0) =\; \mathbf x_0, \quad 
  \forall \mathbf x_0\in\partial\Omega_0^f\backslash\Gamma_0.
\end{alignat}

\subsection{
Semidiscrete scheme: the 
coupling condition treatment}
We first work on the spatial coupling of the two HDG schemes \eqref{hdg-fluid} and
\eqref{hdg-elas} taking into account the interface conditions
\eqref{fsi-ic1}--\eqref{fsi-ic2}.

We define the following quantities 
on the interface $\Gamma_0$:
\[
  \overline{\bld u_h^s}:= \tng(\bld u_h^s)+\nrm(\widetilde{\bld u}_h^s), 
  \quad \overline{\bld d_h}:= \tng(\bld d_h)+\nrm(\widetilde{\bld d}_h), 
  \quad
  \overline{\bld \xi_h}:= \tng(\bld \xi_h)+\nrm(\widetilde{\bld \xi}_h). 
\]
We then define the following coupling term on the deformed interface $\Gamma_t$
\begin{align}
  \label{coupling}
  \mathcal{I}_h\left(
(\bld \epsilon_h^f,
 \widetilde{\sigma}_h^f,
 \widetilde{\bld u}_h^f, 
 \bld u_h^s,
 \widetilde{\bld u}_h^s
 ), 
 (\bld G_h^f,
 \widetilde{\tau}_h^f,
 \widetilde{\bld v}_h^f, 
 \bld \xi_h,
 \widetilde{\bld \xi}_h 
 )\right):= &\;
  \intG{(\overline{\bld u_h^s}\circ\bld \phi_t)\cdot\bld n^f}{\widetilde{\tau}_h^f}
  -\intG{\widetilde{\sigma}_h^f}{(\overline{\bld \xi_h^s}\circ\bld\phi_t)\cdot\bld n^f}\\
       &
       \hspace{-0.3\textwidth}   -2\mu^f\intG{\bld \epsilon_h^f\bld n^f
         }{\tng\left(\widetilde{\bld v}_h^f-
         (\overline{\bld \xi_h^s}\circ\bld\phi_t)\right)
       }
       +2\mu^f\intG{\bld G_h^f\bld n^f}{\tng\left(\widetilde{\bld u}_h^f-
         (\overline{\bld u_h^s}\circ\bld\phi_t)\right)
       }\nonumber\\
       &
       \hspace{-0.05\textwidth}
       +\intG{\alpha^f
         \tng\left(\widetilde{\bld u}_h^f-
         (\overline{\bld u_h^s}\circ\bld\phi_t)
       \right)
         }{
         \tng\left(\widetilde{\bld v}_h^f-
         (\overline{\bld \xi_h^s}\circ\bld\phi_t)
       \right)
       }\nonumber
\end{align}
Note that in the above term, the Nitsche's technique is used to treat the coupling condition 
in the tangential direction, and the mortar method (with $\widetilde{\tau}_h^f$
as the Lagrange multiplier) in the normal direction.

The semidiscrete coupled HDG scheme 
for the FSI model \eqref{ns-eq-reform} and \eqref{elas-eq} with boundary and
interface conditions \eqref{fsi-bc} reads as follows: 
Find 
$
(p_h^f,
 \bld \epsilon_h^f,
 \bld u_h^f,
 \widetilde{\sigma}_h^f,
 \widetilde{\bld u}_h^f)
 \in \Whf{k-1}\times\Shf{k}\times\Vhf{k}\times\tWhf{k}
 \times \tVhf{k}
 $ with $\widetilde{\bld u}_h^f|_{\partial\Omega_t^f\backslash\Gamma_t}=0$,
 and 
$
(\bld P_h,
 \bld F_h,
 \bld d_h,
 \bld u_h^s,
 \widetilde{\bld d}_h,\widetilde{\bld u}_h^s)
 \in \Shs{k}\times\Shs{k}\times\Vhs{k}\times\Vhs{k}
 \times \tVhs{k}\times \tVhs{k}
 $ 
 with 
 $
   \tng({\bld d}_h)|_{\partial\Omega^s\backslash\Gamma_0} =
     \nrm(\widetilde{\bld
   d}_h)|_{\partial\Omega^s\backslash\Gamma_0}=0,
 $ 
such that 
\begin{align}
  \label{hdg-fsi}
  \mathcal{M}_h^f(
  \left.\frac{\partial \bld u_h^f}{\partial t}\right|_{\mathbf x_0} 
  , \bld v_h^f) 
  +
  \mathcal{C}_h^f\left(\bld u_h^f-\bld \omega^f; (\bld u_h^f,\widetilde{\bld u}_h^f),
  (\bld v_h^f,\widetilde{\bld v}_h^f)\right)\quad\quad& \\\nonumber
  +\mathcal{A}_h^f\left(
    (p_h^f, \bld \epsilon_h^f, \bld u_h^f,
    \widetilde{\sigma}_h^f, \widetilde{\bld u}_h^f
    ),
  (q_h^f,\bld G_h^f, \bld v_h^f,
    \widetilde{\tau}_h^f, \widetilde{\bld v}_h^f
)\right)&\\\nonumber
  +\mathcal{M}_h^s\left(\frac{\partial \bld u_h^s}{\partial t},
  \bld \xi_h\right) 
  +\mathcal{A}_h^s\left(
    (\bld P_h, \bld F_h, \bld d_h,
    \widetilde{\bld d}_h),
    (\bld Q_h, \bld G_h, \bld \xi_h,
    \widetilde{\bld \xi}_h)\right)
  \\\nonumber
  +
  \mathcal{I}_h\left((
\bld \epsilon_h^f,
 \widetilde{\sigma}_h^f,
 \widetilde{\bld u}_h^f, 
 \bld u_h^s,
 \widetilde{\bld u}_h^s
 ), 
 (\bld G_h^f,
 \widetilde{\tau}_h^f,
 \widetilde{\bld v}_h^f, 
 \bld \xi_h,
 \widetilde{\bld \xi}_h 
 )\right)&
  \;= \mathcal{F}_h^f(\bld v_h^f)+
\mathcal{F}_h^s(\bld \xi_h),
\end{align}
for all $
(q_h^f,
 \bld G_h^f,
 \bld v_h^f,
 \widetilde{\tau}_h^f,
 \widetilde{\bld v}_h^f)
 \in \Whf{k-1}\times\Shf{k}\times\Vhf{k}\times\tWhf{k}
 \times \tVhf{k}
 $ with $\widetilde{\bld v}_h^f|_{\partial\Omega_t^f\backslash\Gamma_t}=0$,
 and 
$
(\bld Q_h,
 \bld G_h,
 \bld \xi_h,
 \widetilde{\bld \xi}_h
 )
 \in \Shs{k}\times\Shs{k}\times\Vhs{k}
 \times \tVhs{k}
 $ 
 with 
 $
   \tng({\bld \xi}_h)|_{\partial\Omega^s\backslash\Gamma_0} =
     \nrm(\widetilde{\bld
   \xi}_h)|_{\partial\Omega^s\backslash\Gamma_0}=0,
 $
 where the structure displacements and velocities satisfy the relation
 \eqref{vel-disp}.


 Stability of the scheme \eqref{hdg-fsi} is documented in the following result.
 \begin{theorem}
Let 
$
(p_h^f,
 \bld \epsilon_h^f,
 \bld u_h^f,
 \widetilde{\sigma}_h^f,
 \widetilde{\bld u}_h^f)
 \in \Whf{k-1}\times\Shf{k}\times\Vhf{k}\times\tWhf{k}
 \times \tVhf{k}
 $ 
 and 
$
(\bld P_h,
 \bld F_h,
 \bld d_h,
 \bld u_h^s,
 \widetilde{\bld d}_h,\widetilde{\bld u}_h^s)
 \in \Shs{k}\times\Shs{k}\times\Vhs{k}\times\Vhs{k}
 \times \tVhs{k}\times \tVhs{k}
 $ be the solution to the semi-discrete scheme \eqref{hdg-fsi}. 
Then there holds
\begin{align}
  \label{ener-fsi}
& \frac{\partial}{\partial t}
\left(
  \frac12\intV{\rho^f\bld u_h^f}{\bld u_h^f}{f}+
\mathsf{E}_h^s\right)
 +
 2\mu^f\intV{\bld\epsilon_h^f}{\bld\epsilon_h^f}{f} 
 +\intF{\gamma_h
   \left|\tng\left(\bld u_h^f-\widetilde{\bld u}_h^f\right)
       \right|^2
     }{1}{f}
   \\ \nonumber
 & 
 +\intG{\alpha^f
         \left|\tng\left(\widetilde{\bld u}_h^f-
         (\overline{\bld u_h^s}\circ\bld\phi_t)
       \right)\right|^2
         }{1}
     =\mathcal{F}_h^f(\bld u_h^f)
     +\mathcal{F}_h^s(\bld u_h^s).
\end{align}
where 
the elastic energy $\mathsf{E}_h^s$ is given in \eqref{elas-ener}.
 \end{theorem}
\begin{proof}
  Taking test functions in the coupling term \eqref{coupling} to be the trial
  functions, we easily get
  \[
  \mathcal{C}_h\left(
(\bld \epsilon_h^f,
 \widetilde{\sigma}_h^f,
 \widetilde{\bld u}_h^f, 
 \bld u_h^s,
 \widetilde{\bld u}_h^s
 ), 
(\bld \epsilon_h^f,
 \widetilde{\sigma}_h^f,
 \widetilde{\bld u}_h^f, 
 \bld u_h^s,
 \widetilde{\bld u}_h^s
 ) 
 \right):= 
\intG{\alpha^f
        \left| \tng\left(\widetilde{\bld u}_h^f-
         (\overline{\bld u_h^s}\circ\bld\phi_t)
       \right)\right|^2
         }{
        1 
       }.
  \]
  The energy identity \eqref{ener-fsi} is simply obtained by 
  combining the above identity 
  with the stability results in Theorem \ref{thm:stability} and Theorem 
  \ref{thm:elas}.
\end{proof}

\subsection{Fully discrete scheme: the moving mesh algorithm}
Following the previous two sections, we consider a BDF2 temporal discretization.
The fully discrete scheme is obtained from the equations \eqref{hdg-fsi}
by replacing the term $
  \left.\frac{\partial \bld u_h^f}{\partial t}\right|_{\mathbf x_0}
$ by $D_t^2\bld u_h^{f,n}$ in \eqref{ale-bdf2},  and the term 
$
\frac{\partial \bld u_h^s}{\partial t}
$
by $\mathsf{D}_t^2\bld u_h^{s,n}$ in \eqref{bdf2-s1}, 
and evaluating the other terms at time level $t^n$, 
where the structure displacements are given explicitly by the formulas
\eqref{disp-X}. We skip the complete formulation.

The last ingredient for an implementable scheme is the construction 
of the (unknown) ALE map \eqref{ale}, which shall satisfy the boundary
condition \eqref{fsi-map}, and, ideally,  can be used to handle large mesh deformations.
Various moving mesh algorithms exist in the literature \cite{Wick11,Shamanskiy20}.
Here we adopt the 
nonlinear elasticity model
with a mesh-Jacobian-based stiffening proposed in \cite{Shamanskiy20}.
In particular, we compute the ALE map $\bld \phi_{h}^n
\in \bld S_h^k$ 
at time $t^n$ 
such that it satisfies the 
boundary conditions
\begin{subequations}
  \label{ale-maps}
  \begin{alignat}{2}
  \label{ale-bc}
  \bld\phi_h^n(\mathbf x_0) = &\;\mathbf x_0 +\overline{ \bld d_h^n}(\mathbf x_0),&&\quad 
  \forall \mathbf x_0\in\Gamma_0,\quad\quad 
  \bld\phi_h^n(\mathbf x_0) =\; \mathbf x_0, \quad 
  \forall \mathbf x_0\in\partial\Omega_0^f\backslash\Gamma_0,
\end{alignat}
and 
\begin{align}
  \label{ale-map}
  \intVX{\frac{1}{J_K^0}\bld P_a}{\nabla_{\mathbf x_0}\bld \psi_h} = 0,
\quad \forall 
\bld \psi_h\in \bld S_h^k \text{ with } \bld \psi_h|_{\partial
\Omega_0^f}=0,
\end{align}
\end{subequations}
where $J_K^0$ is the Jacobian determinant for the initial configuration
given in \eqref{jac}, and 
following \cite{Shamanskiy20}, we
use the
logarithmic variation of the neo-Hookean material law
\[
  \bld P_a:= \bld F_a(
  \lambda_a\mathrm{ln}J_a\, \bld C_a^{-1}
  +\mu_a(\mathbf I-\bld C_a^{-1})
  )
\] 
where 
$
  \bld F_a=\nabla_{\mathbf x_0}\bld \phi_h, \;\;J_a=\mathrm{det}(\bld F_a),\;\; \bld C_a = \bld F_a^T\bld F_a.
$
We take the artificial  Lam\'e parameters $\mu_a=1$ and $\lambda_a = 1.5$
so that the Poisson ration $\nu_a=\frac{\lambda_a}{2(\lambda_a+\mu_a)}=0.3$. 
After the discrete ALE maps  $\bld \phi_h^n$ are computed, we again use BDF2 to 
approximate the 
the mesh velocity: 
\begin{align}
  \label{mesh-vel}
  \bld \omega^{f,n}_h := 
  \frac{1.5\bld \phi_h^{n}-2\bld \phi_h^{n-1}+0.5\bld \phi_h^{n-2}}{\Delta
  t}.
\end{align}

To simplify the calculation, we further decouple the computation of the ALE map \eqref{ale-maps} with the
HDG solver \eqref{hdg-fsi} for the field unknowns
by extrapolating the boundary data \eqref{ale-bc}. Hence, the geometric nonlinearity
is treated explicitly.
Moreover, 
following \cite{Shamanskiy20},
we solve the nonlinear equations \eqref{ale-maps} approximately by 
only solving a linearized problem of \eqref{ale-maps} 
around the previous state
$\bld \phi_h^{n-1}$. 
The algorithm to advance one time step is sketch below.

\textbf{Algorithm 1: } 
{\sf
  Given the ALE map $\bld \phi_h^j$ and the quantities $\bld u_h^{f,j}, \bld u_h^{s,j}, \bld d_h^{j}, 
  \widetilde{\bld d}_h^{j}$ 
  for $j=n-2$ and $n-1$,
  find the solution 
  at time $t^n$ by the following steps:
\begin{itemize}
  \item[(1)] 
    Replace the (implicit) boundary data 
$
 \overline{ \bld d_h^{n}}$ 
    in \eqref{ale-bc} by 
the second-order extrapolation $
 2\overline{ \bld d_h^{n-1}}
 -\overline{ \bld d_h^{n-2}}.
 $ 
 Then 
 solve a linearized version of the 
 eqauations \eqref{ale-maps} to obtain $\bld \phi_h^n$.
In particular, 
we linearize the equations \eqref{ale-maps} around the previous
 state $\bld \phi_h^{n-1}$.
  \item[(2)] 
Deform the fluid mesh $\Th^{f,t}= \bld\phi_h^n(\Th^{f,0})$, and 
set the mesh velocity $\bld \omega_h^{f,n}$ using equation \eqref{mesh-vel}.
Solve for the velocities
$\bld u_h^{f,n}, \bld u_h^{s,n},\widetilde{ \bld u}_h^{s,n}$ 
using the HDG scheme \eqref{hdg-fsi} coupled with the BDF2 time stepping.
Then update displacements $\bld d_h^{n},\widetilde{ \bld d}_h^{n}$ 
using equations \eqref{disp-X}.
\end{itemize}
}
The major computational cost of the above algorithm lies in the HDG solver in
step 2, which takes more than 90\% of the overall run time.

\section{Numerical results}
\label{sec:num}
We test the performance of 
\textbf{Algorithm 1} on the classical benchmark problem proposed by 
Turek and Hron \cite{Turek06} where reference data are available in \cite{FE}.
Our numerical simulations are performed using the open-source finite-element software 
{\sf NGSolve} \cite{Schoberl16}, \url{https://ngsolve.org/}.

\subsection{Problem setup}
The problem is a two-dimensional incompressbile channel flow around a rigid
cylinder with an attached nonlinearly elastic bar.
The domain is depicted in Figure~\ref{fig:domain}.
\begin{itemize}
  \item The domain dimensions are: length $L=2.5$, height $H=0.41$.
  \item The circle center is positioned at $C=(0.2,0.2)$ (measured from the left bottom
corner of the channel) and the radius is $r=0.05$.
\item The elastic structure bar has length $l=0.35$ and height $h=0.02$, the right
bottom corner is positioned at $(0.6,0.19)$, and the left end is fully attached to the fixed cylinder.
\item The control point is $A(t)$, fixed with the structure with $A(0)=(0.6,0.2)$.
\end{itemize}
\begin{figure}[h!]
\centering
\includegraphics[width=.45\textwidth]{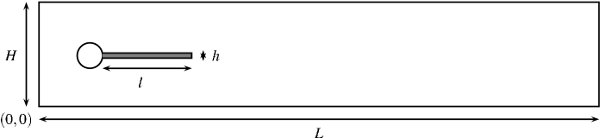}\quad\quad\quad
\includegraphics[width=.45\textwidth]{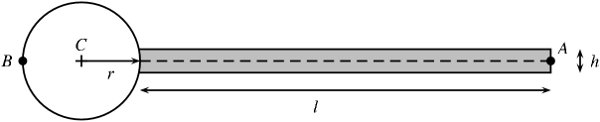}
\caption{The domain for the FSI problem \cite{Turek06}.}
\label{fig:domain}
\end{figure} 
The fluid region is governed by the Navier-Stokes equations
\eqref{ns-eq-reform} with $\bld f^f=0$, and the 
elastic structure is governed by the equations for hyperelasticity \eqref{elas-eq}
with $\bld f^s=0$.
The 
coupling conditions \eqref{fsi-ic1}--\eqref{fsi-ic2} is used on the
fluid-structure interface $\Gamma_t$, and the following boundary conditions are used:
\begin{itemize}
  \item  
A parabolic velocity profile is prescribed at the left channel inflow
\[
  \bld u^f (0,y, t) =\left\{
    \begin{tabular}{ll}
   $ \bld u^f(0, y)\frac{1-\cos(\frac{\pi}{2}t)}{2} $
   & if $t<2$,\\[2ex]
   $ \bld u^f(0, y) $
   & otherwise,
    \end{tabular}
\right.
\]
where $\bld u^f(0,y)= 1.5\bar U \frac{y(H-y)}{({H}/{2})^2}=
61.5\bar U \frac{4.0}{0.1681}{y(0.41-y)}
$. 
\item 
  The stress-free boundary condition $\bld\sigma^f\bld n=0$
is prescribed at the outflow.
\item The no-slip condition 
  ($\bld u^f=0$, or $\bld d=\bld u^s=0$)
  is prescribed on all the other boundary
  parts. 
\end{itemize}

Two test cases resulting in time periodic solutions are considered, which are 
denoted as FSI2 and FSI3 in \cite{Turek06}.
The associated material parameters are listed in Table~\ref{table:mat}.
\begin{table}[ht!]
\begin{center}
  \scalebox{1}
  {
  \begin{tabular}{ c | c | c|c|c|c|c } 
    parameter &
    $\rho^s[10^3\mathrm{\frac{kg}{m^3}}]$
              &$\lambda^s[10^6\mathrm{\frac{kg}{ms^2}}]$ & $\mu^s[10^6\mathrm{\frac{kg}{ms^2}}]$
              &
    $\rho^f[10^3\mathrm{\frac{kg}{m^3}}]$
              & $\mu^f[\mathrm{\frac{kg}{ms}}]$
              &$\bar{U}[\mathrm{\frac{m}{s}}]$
  \\\hline
    FSI2 & 10 & 2.0 & 0.5 & 1 &1 & 1\\
    FSI3 & 1 & 8.0 & 2.0 & 1 &1 & 2\\[.1ex]
\end{tabular}
}
\end{center}
  \caption{\it Parameter settings for the two test cases.}
\label{table:mat}
\end{table}

Quantities of interest are
\begin{itemize}
  \item The displacement of the control point $A$ at the end of the beam
    structure (see Figure~\ref{fig:domain}).
  \item The lift and drag forces acting on the cylinder and the beam structure:
    \[
      (F_D,F_L) = \int_S\bld \sigma^f\bld n\,\mathrm{ds},
    \] 
    where $S$ denotes the boundary between the fluid domain and the cylinder together with the elastic
structure. 
\end{itemize}
We compare the computational results with the reference data provided in 
\cite{FE}
after a fully developed periodic flow is formed.

\subsection{Discretization}
We apply \textbf{Algorithm 1} to both test cases, and take the polynomial degree
$k=3$ throughout. 
Three set of meshes  are used in the simulation.
The coarse mesh, which contains $495$ triangular elements,  is generated automatically from the geometry by
NETGEN \cite{Schoberl99}. Curved elements with polynomial degree $k=3$ are used
for elements near the cylinder.
The intermediate and fine meshes are obtained from the coarse mesh by uniform
refinements. See Figure~\ref{fig:mesh} for the coarse and intermediate meshes.
\begin{figure}[h!]
\centering
\includegraphics[width=.45\textwidth]{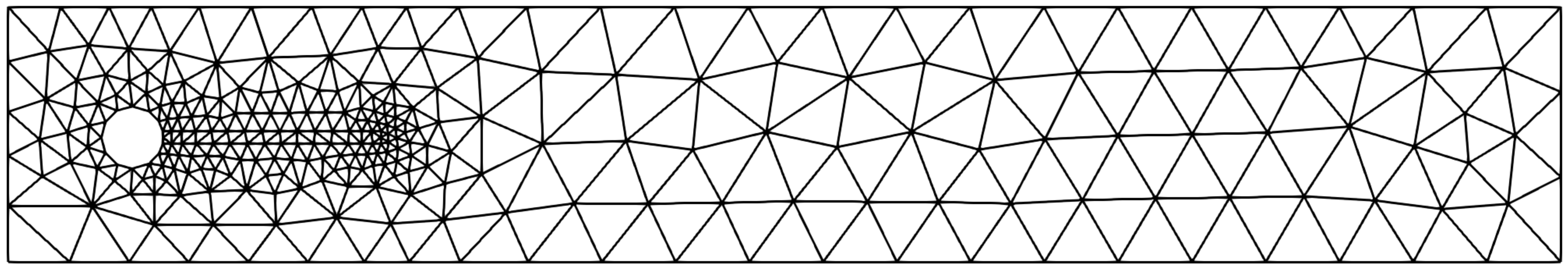}
\includegraphics[width=.45\textwidth]{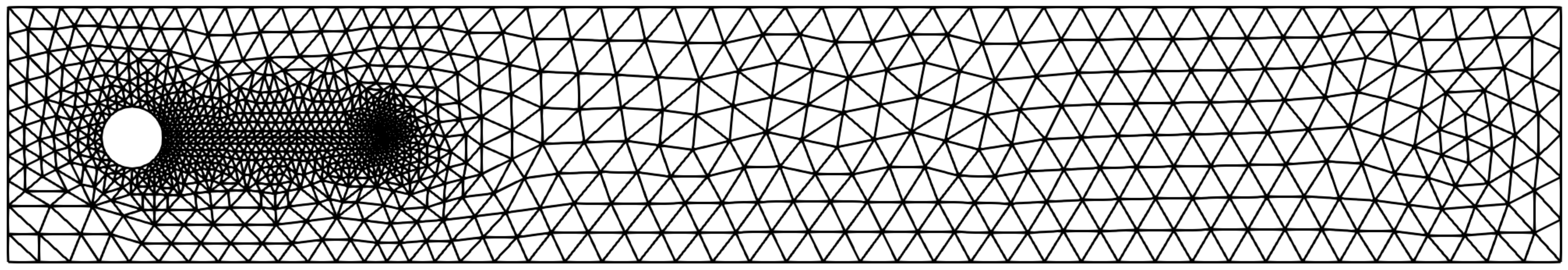}
\caption{The computational meshes. Left: coarse mesh with 495 elements. 
  Right:
intermediate mesh with 1980 elements.
Curved elements with polynomial degree $3$ are used near the cylinder.
}
\label{fig:mesh}
\end{figure}

The nonlinear system in \eqref{hdg-fsi} is solved via the Newton's method where
the iteration is terminated 
with an absolute tolerance of $10^{-8}$ for the $l_2$-norm of the residual
vector.
Static condensation is used for the resulting linear system problems, where 
the globally coupled degrees of freedom consists of {\it only} those lie on the
mesh skeleton, namely, those for the normal-normal component of the fluid stress 
$\widetilde{\sigma}_h^f$, the tangential component of the fluid velocity
$\widetilde{\bld u}_h^f$, the normal component of the structure velocity 
$\widetilde{\bld u}_h^s$, and the structure velocity $\bld u_h^{s}$ on the mesh skeleton.
The globally coupled linear system is then solved via a sparse direct solver.
We observe in average 4-6 iterations for the nonlinear solver to converge. 
The number of globally coupled degrees of freedom 
is $6188$ on the coarse mesh, $24256$ on the intermediate mesh, and 
$96032$ on the fine mesh.

For the FSI2 problem, we 
use three different time step size, namely, $\Delta t = 0.02, 0.01, 0.005$, and 
stop the simulation at time $T=15$.
A fully developed periodic flow is observed starting around 
time $t=10$.
In Figure~\ref{fig:fsi2}, we plot the quantities of interest for the fully
developed flow over one second time for \textbf{Algorithm 1} on the 
three  meshes, 
along
with reference data provided in
\cite{FE}. 
Two comments are in order.
\begin{itemize}
  \item [(1)]
    On a fixed mesh, 
    we observe a better phase and amplitude accuracy, compared with the
    reference data, when decreasing the time step size from  $\Delta
    t = 0.02$ (in red) 
    to $\Delta
    t = 0.01$ (in green).
    However,  further decreasing the time step size to $\Delta t=0.005$ (in
    blue) does not lead to significantly better results.
  \item [(2)]
    Comparing the results on different meshes when the time step size is fixed,
  we find  the all results are very similar to each other. 
    This indicates that our spatial discretization is very accurate even on the coarse mesh.
\end{itemize}

\begin{figure}[h!]
\centering
\includegraphics[width=\textwidth]{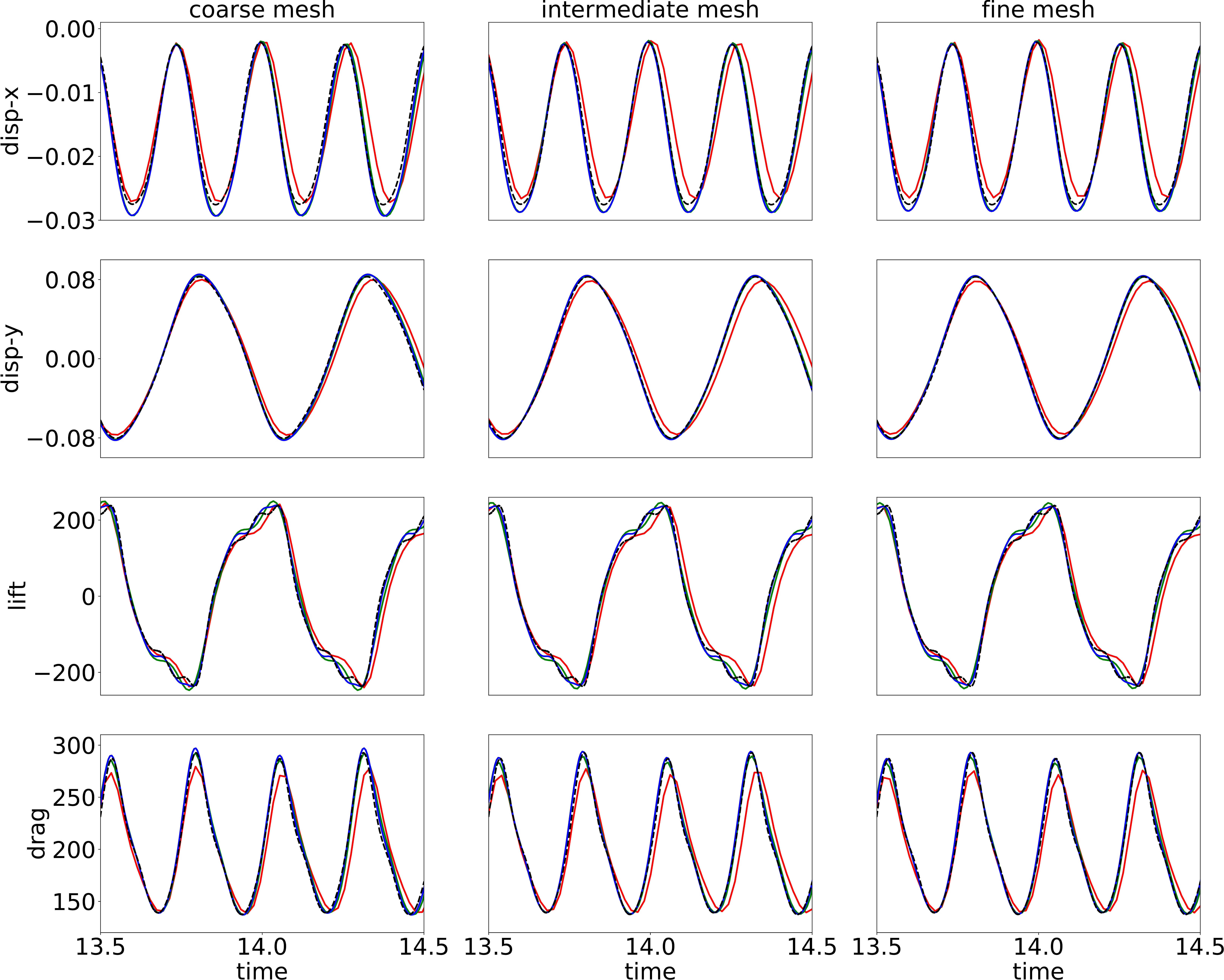}
\caption{
 The quantities of interest for FSI2 in the fully developed flow regime. 
 Left column: coarse mesh. Middle column: intermediate mesh. Right column: fine
 mesh.  
First row: $x$-component of displacement at point $A$.
Second row : $y$-component of displacement at point $A$.
Third row: lift force.
Last row: drag force.
 Red line: $\Delta t = 0.02$;
Green line: $\Delta t = 0.01$;
Blue line: $\Delta t = 0.005$;
Dashed black line: reference solution from \cite{FE}.
}
\label{fig:fsi2}
\end{figure}

For the FSI3 problem, we 
use three different time step size, namely, $\Delta t = 0.01, 0.005, 0.0025$, and 
stop the simulation at time $T=9$.
A fully developed periodic flow is observed starting around 
time $t=5$.
In Figure~\ref{fig:fsi3}, we plot the computed quantities of interest for the fully
developed flow over half second time.
Overall, all the results are in very good agreements with
the reference data. We make two more comments.
\begin{itemize}
  \item [(1)]
    Similar to the FSI2 case, 
    on a fixed mesh, 
    we observe a better phase and amplitude accuracy, compared with the
    reference data, when decreasing the time step size from  $\Delta
    t = 0.01$ (in red) 
    to $\Delta
    t = 0.005$ (in green).
    However,  further decreasing the time step size to $\Delta t=0.0025$ (in
    blue) does not lead to better results.
  \item [(2)]
    Comparing the results on different meshes when the time step size is fixed
    at  $\Delta t = 0.005$ (in green) or 
  $\Delta t = 0.0025$ (in blue), 
  we find  the results in the displacements and lift force are very similar on the
  intermediate and fine meshes, which are slightly better than those on the
  coarse mesh.
  While the results in the drag force on the fine mesh is still the most
  accurate,  those on the coarse mesh  is surprisingly better than 
  those on the intermediate mesh, which slightly
  underestimate the magnitude by about $1\%$.
\end{itemize}

\begin{figure}[h!]
\centering
\includegraphics[width=\textwidth]{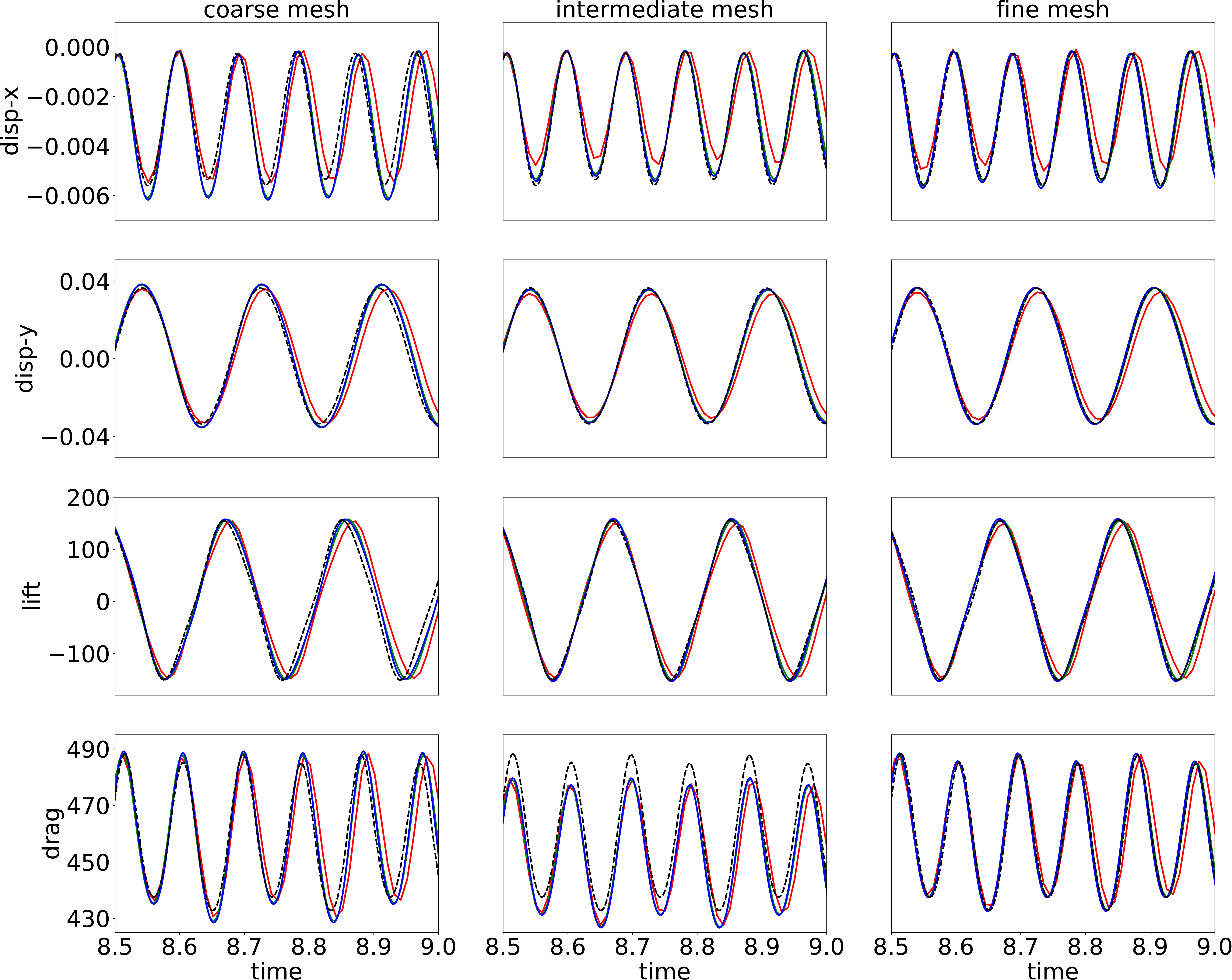}
\caption{
 The quantities of interest for FSI3 in the fully developed flow regime. 
 Left column: coarse mesh. Middle column: intermediate mesh. Right column: fine
 mesh.  
First row: $x$-component of displacement at point $A$.
Second row : $y$-component of displacement at point $A$.
Third row: lift force.
Last row: drag force.
 Red line: $\Delta t = 0.02$;
Green line: $\Delta t = 0.01$;
Blue line: $\Delta t = 0.005$;
Dashed black line: reference solution from \cite{FE}.
}
\label{fig:fsi3}
\end{figure} 

\section{Conclusion}
\label{sec:conclude}
We have presented a novel coupled HDG scheme for the FSI problem, where the
moving domain fluid Navier-Stokes equations are solved using an
ALE-divergence-free HDG scheme, and the equations for nonlinear elasticity is
solved in the Lagrangian framework using a novel $H(\mathrm{curl})$-conforming
HDG scheme.
Numerical results on the benchmark problem of Turek and Hron \cite{Turek06}
showed good performance of our method.
Our ongoing work consists of 
(1) the construction of good preconditioners for the linearized problems,
and  (2) the construction of efficient, robust and accurate partitioned
algorithms for this HDG scheme.

\bibliographystyle{siam}

\end{document}